\documentclass[a4paper]{article}

\usepackage{amsmath, amsthm, amsfonts, amssymb, accents}
\usepackage{graphicx,color}
\usepackage{srcltx}
\usepackage{dsfont}

\theoremstyle{plain}
\newtheorem{theorem}{Theorem}[section]

\newtheorem{lemma}[theorem]{Lemma}

\newtheorem{corollary}{Corollary}

\theoremstyle{remark}

\numberwithin{equation}{section}

\def\tht{\theta}
\def\Om{\Omega}

\def\e{\varepsilon}
\def\g{\gamma}
\def\G{\Gamma}
\def\l{\lambda}
\def\p{\partial}
\def\D{\Delta}

\def\k{\varkappa}
\def\E{\mbox{\rm e}}
\def\a{\alpha}
\def\b{\beta}

\def\d{\delta}
\def\L{\Lambda}
\def\z{\zeta}
\def\vs{\varsigma}

\def\vp{\varphi}

\def\vt{\vartheta}

\def\Odr{\mathcal{O}}
\def\H{W_2}

\def\Hinf{W_\infty}
\def\Hper{W_{2,per}}
\def\Ho{W_{2,0}}

\def\di{\,d}

\def\I{\mathrm{I}}
\def\iu{\mathrm{i}}

\def\Gp{\mathring{\G}}
\def\gp{\mathring{\g}}

\def\Hpe{\mathring{\mathcal{H}}_\e}

\def\hpe{\mathring{\mathfrak{h}}_\e}

\def\Hoper{\mathring{W}_{2,per}}

\def\Ups{\Upsilon}

\def\po{\mathring{\psi}}

\def\pex{\mathring{\psi}_\e^{\mathrm{ex}}}
\def\pbl{\mathring{\psi}_\e^{\mathrm{bl}}}

\def\Psin{\mathring{\Psi}^{\mathrm{in}}}
\def\psin{\mathring{\psi}^{\mathrm{in}}}
\def\go{\mathring{\g}^1}
\def\Go{\mathring{\G}^1}
\def\Pso{\mathring{\Psi}_\e}

\DeclareMathOperator{\RE}{Re}
\DeclareMathOperator{\IM}{Im} \DeclareMathOperator{\spec}{\sigma}

\DeclareMathOperator{\essspec}{\sigma_{e}}

\newcounter{assumption}
\begin{document}
\allowdisplaybreaks

\title{\textbf{Waveguide with non-periodically alternating Dirichlet and Robin conditions: homogenization and asymptotics.}}
\author{Denis Borisov\,$^a$, Renata Bunoiu$^b$, Giuseppe Cardone$^c$}

\date{\empty}

\maketitle 
\begin{center}
\begin{quote}
\begin{enumerate}
{\it
\item[$a)$]
Bashkir State Pedagogical University,
October St.~3a, 450000 Ufa,
Russian Federation; \texttt{borisovdi@yandex.ru}
\item[$b)$]
LMAM, UMR 7122, Universit\'{e} de Lorraine et CNRS, Ile du Saulcy, F-57045 METZ Cedex 1, France; \texttt{bunoiu@univ-metz.fr}
\item[$c)$]
University of Sannio,
Department of Engineering, Corso Garibaldi,
107, 82100 Benevento, Italy; \texttt{giuseppe.cardone@unisannio.it}
}
\end{enumerate}
\end{quote}
\end{center}

\begin{abstract}
We consider a magnetic Schr\"odinger operator in a planar infinite strip with frequently and non-periodically alternating Dirichlet and Robin boundary conditions. Assuming that the homogenized boundary condition is the Dirichlet or the Robin one, we establish the uniform resolvent convergence in various operator norms and we prove the estimates for the rates of convergence. It is shown that these estimates can be improved by using special boundary correctors. In the case of periodic alternation, pure Laplacian, and the homogenized Robin boundary condition, we construct two-terms asymptotics for the first band functions, as well as the complete asymptotics expansion (up to an exponentially small term) for the bottom of the band spectrum.
\end{abstract}
%
 
%---------------------%
\section{Introduction}

In the present paper we study a magnetic Schr\"odinger operator in an infinite planar strip with frequently non-periodically alternating Dirichlet and Robin boundary conditions, cf. fig.~\ref{fig1}. This model was formulated first in \cite{BC} for the pure Laplacian with periodically alternating Dirichlet and Neumann boundary conditions. The further studies were done in \cite{BBC1}, \cite{BBC-CR}, \cite{BBC-PMA}. The motivation for such studies as well as the reviews of previous results were done in all the details in \cite{BBC1} and \cite{BC} and here we just briefly remind it.

The motivation is threefold and comes from the waveguide theory and the homogenization theory. In the former one of the popular models is %D% that of
the waveguide with windows. Usually the waveguide is modeled by a planar strip or layer, where an elliptic operator is considered. The windows are modeled by a hole cut out on the boundary of the waveguide and coupling it with another waveguide. In the symmetric case the hole can be replaced by a segment on the boundary at which one switches the type of the boundary condition. Such models in the case of one or several finite windows were studied by various authors, see, for instance, \cite{Win2}, \cite{Win1},  \cite{Win3}, \cite{Win4}, \cite{Win5}, \cite{Win6}, \cite{Win7}. In our model the windows are modeled by the segments on the boundary with a general Robin condition. Each segment is finite, while their total number is infinite. Exactly the last fact %D%one
is the main difference of our model in comparison with those in \cite{Win2}, \cite{Win1}, \cite{Win3}, \cite{Win4}, \cite{Win5}, \cite{Win6}, \cite{Win7}.

The second reason to study our model comes from the series of papers devoted to the problems in bounded domains with frequent alternation of boundary conditions. Not trying to cite all of them, we mention just \cite{Al1}, \cite{B03}, \cite{AHP-4}, \cite{Al3}, \cite{Al4}, \cite{Al2}, \cite{Al7},     \cite{Al6}, \cite{Al5}, \cite{AHP-23}, \cite{G-SPMJ98}, \cite{AHP-25}, \cite{AHP-24}. So, it was natural to study the case of an unbounded domain.

The third and the main part of the motivation is the recent series of papers by M.Sh. Birman, T.A. Suslina on one hand and by V.V. Zhikov and S.E. Pastukhova on the other -- see, for instance, \cite{Bir1}, \cite{BS4}, \cite{BS2}, \cite{BS5}, \cite{Su1}, \cite{Pas}, \cite{PT}, \cite{Zh3}, \cite{Zh4}, \cite{Zh5}, \cite{CPZh} and further papers of these authors. In these papers the homogenization of the differential operators with fast oscillating coefficients in unbounded domains was studied. The operators were considered as unbounded ones in appropriate Hilbert spaces. The main result was the proof of the uniform resolvent convergence as well as the leading terms of the asymptotics for the perturbed resolvents in the sense of the operator norms. The last cited series considered the problems lying in the intersection of the homogenization theory and the spectral theory of unbounded operators. The approaches of both the theories were involved and it happened to be fruitful and interesting. From this point of view our model continues the ideas of papers by M.Sh. Birman, T.A. Suslina, V.V. Zhikov, S.E. Pastukhova, since we consider the problem from operator point view, but with the boundary geometric perturbation from the homogenization theory.

\begin{figure}
\begin{center}
\includegraphics[scale=0.65]{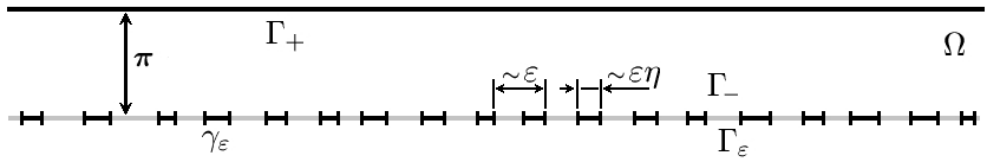}
\caption{Waveguide with non-periodically alternating boundary conditions}\label{fig1}
\end{center}
\end{figure}

In comparisons with our previous works \cite{BBC1}, \cite{BBC-CR}, \cite{BBC-PMA}, \cite{BC}, the present one has several advantages. While in the cited papers we considered the pure Laplacian only, now we deal with a much more general operator involving variable coefficients. We also consider  alternating Dirichlet and Robin boundary conditions instead of Dirichlet and Neumann ones in \cite{BBC1}, \cite{BBC-CR}, \cite{BBC-PMA}, \cite{BC}. Moreover, now the homogenized operator involves Robin boundary condition instead of Dirichlet or Neumann ones. More precisely, we have an additional term in the coefficient in the homogenized Robin condition, and this term appears due to the geometric structure of the perturbation. Such situation was not considered before and it is the most complicated among the possible ones. One more important difference is that in the present paper we consider \emph{non-periodic} alternation of boundary conditions, while in \cite{BBC1}, \cite{BBC-CR}, \cite{BBC-PMA}, \cite{BC} only strictly periodic alternations were studied. Our assumptions for the structure of alternations are quite weak and include quite a wide class of possible non-periodic alternations. We prove the uniform resolvent convergence in such case when the homogenized condition is the Robin one and the same result for the homogenized Dirichlet condition. The estimates for the rate of convergence are given in two possible operator norms, see Theorems~\ref{th1.1},~\ref{th1.4}. We also find that in the case of the homogenized Robin condition the uniform resolvent convergence does not hold in the sense of the norm of the operators acting from $L_2$ to $\H^1$, while it was true in the case of homogenized Dirichlet and Neumann condition in \cite{BBC1}, \cite{BBC-CR}, \cite{BBC-PMA}, \cite{BC}.

To obtain the described results, we use the combination of the techniques employed in \cite{BBC1}, \cite{BC} for periodic case and that from \cite{B03}. Exactly in the last paper the non-periodic alternation in the bounded domain was treated, and we adapt these ideas for our case of the unbounded domain. In \cite{B03} the main result were the asymptotic estimates and the asymptotic expansions for the eigenvalues of the pure Laplacian, while the resolvent convergence was not considered. It should be also said that it is the first time in the problems on boundary homogenization that for non-periodic alternation the uniform resolvent convergence is proven and the estimates for the rates of convergence are provided.

The next part of our main result concerns the description of the asymptotic behavior of the spectrum. Since in the case of  general operator with non-periodic alternation the structure of the spectrum can be very complicated, we restrict ourselves to the case of pure Laplacian with periodic alternation. In this case we can apply Floquet-Bloch decomposition and get the band spectrum. Our main results are two-terms asymptotics for the first band functions and the complete asymptotic expansion (up to an exponentially small term) for the bottom of the spectrum. Similar results were obtained in \cite{BBC1}, \cite{BBC-CR}, \cite{BBC-PMA}, \cite{BC} for the case of homogenized Dirichlet or Neumann condition. In the present paper we prove them in the case of homogenized Robin condition assuming as above that the geometric structure of the alternation generates additional term in the homogenized condition.

In conclusion of this section, let us describe briefly the structure of the paper. In the next section we formulate the problem and the main results. We also discuss open problems, and, in particular, the Bethe-Sommerfeld conjecture for our model, as well as possible ways of treating them. In the third section we prove the uniform resolvent convergence in the case of the homogenized Robin condition and establish the estimates for the rates of convergence. Similar results but for the homogenized Dirichlet condition are proven in the fourth section. In the fifth section we study the band functions in the periodic case and we prove two-terms asymptotics for them. In the last sixth section we construct the complete asymptotic expansion for the bottom of the essential spectrum in the case of periodic alternation and the homogenized Robin condition.

%------------------------%
\section{Formulation of the problem and the main result}\label{Sec.main}
%------------------------%

%
Let $x=(x_1,x_2)$ be the Cartesian coordinates in $\mathds{R}^2$,  $\Om:=\{x: 0<x_2<\pi\}$ be an infinite strip of the width $\pi$, $\e$ be a small positive parameter.
By $\G_+$ and $\G_-$ we denote respectively the upper and lower boundary of $\Om$. On the lower boundary we introduce an infinite set of the points $(s_j^\e,0)$, $j\in\mathds{Z}$, $s_0^\e=0$, $s_j^\e<s_{j+1}^\e$ such that the distance between each two neighboring points is of order $\Odr(1)$. By $a_j^+(\e)$, $a_j^-(\e)$, $j\in\mathds{Z}$, we denote two sets of functions such that
\begin{equation}\label{1.23}
s_{j-1}^\e+a_{j-1}^+(\e)<s_j^\e-a_j^-(\e),\quad j\in \mathds{Z}.
\end{equation}
Employing the introduced points and functions, we partition the lower boundary $\G_-$ as  follows,
\begin{equation*}
\g_\e:=\bigcup\limits_{j=-\infty}^{+\infty}\{x: \e s_j^\e-\e a_j^-(\e)<x_1<\e s_j^\e+\e a_j^+(\e),
\, x_2=0,\, j\in \mathds{Z}\},
\quad \G_\e:=\G_-\setminus\overline{\g_\e}.
\end{equation*}

In this paper we study a magnetic Schr\"odinger operator in $\Om$ subject to the Dirichlet boundary condition on $\G_+\cup\g_\e$ and to the Robin condition on $\G_\e$. We define the differential expression corresponding to the operator as
\begin{equation}\label{1.2}
\mathcal{H}_\e:=(\iu\nabla +A)^2+V\quad\text{in}\quad L_2(\Om),
\end{equation}
where $A=A(x)=(A_1(x),A_2(x))$ is a magnetic potential, and $V=V(x)$ is an electric potential. We assume that $A_i\in\Hinf^1(\Om)$, $i=1,2$, $V\in L_\infty(\Om)$ and these functions are real-valued. The boundary conditions for the operator $\mathcal{H}_\e$ are introduced as
\begin{equation*}%\l%abel{1.3}
u=0\quad\text{on}\quad\G_+\cup\g_\e,\qquad \left(-\frac{\p}{\p x_2}-\iu A_2+b\right)u=0\quad\text{on}\quad\G_\e,
\end{equation*}
where $A_2$ is supposed to be $A_2(x_1,0)$, and $b\in\Hinf^1(\mathds{R})$ is a real-valued function. Rigorously we introduce  $\mathcal{H}_\e$ as the self-adjoint operator in $L_2(\Om)$ associated with the symmetric lower-semibounded closed sesquilinear form
\begin{equation}\label{1.4}
\mathfrak{h}_\e[u,v]:=\big((\iu\nabla+A)u,(\iu\nabla+A) v\big)_{L_2(\Om)} +(Vu,v)_{L_2(\Om)}+
(bu,v)_{L_2(\Gamma_{\e})}\quad
\text{on}\quad \Ho^1(\Om,\G_+\cup\g_\e).
\end{equation}
Hereinafter we denote by $\Ho^1(Q,S)$  the set of the functions in $\H^1(Q)$ having the zero trace on the manifold $S$ lying in the closure of the domain $Q$. Our main aim is to study the asymptotic behavior  of the resolvent and the spectrum of $\mathcal{H}_\e$ as $\e\to+0$.

Before presenting the main results, we make certain assumptions on the structure of the alternation of the boundary conditions in the operator $\mathcal{H}_\e$. The first assumption describes the distribution of the parts of %D%the set
$\g_\e$.

\begin{enumerate}\def\theenumi{(A\arabic{enumi})}
\item\label{asA1} There exists a function $\vt_\e=\vt_\e(s)\in C^3(\mathds{R})$ such that
    \begin{align}
    &\vt_\e(\e s_j^\e)=\e\pi j, \label{1.19}
    \\
    & c_1^{-1}\leqslant \vt_\e'(s)\leqslant c_1,\hphantom{\leqslant c_1}\quad s\in \mathds{R}\label{1.20}
    \\
    &|\vt_\e''(s)|+|\vt_\e'''(s)|\leqslant c_1,\quad s\in\mathds{R},\label{1.21}
    \end{align}
    where $c_1$ is a positive constant independent of $s$ and $\e$.
\end{enumerate}

The second assumption concerns the lengths of  the parts of %D%the set
$\g_\e$.

\begin{enumerate}\def\theenumi{(A\arabic{enumi})} \setcounter{enumi}{1}
\item\label{asA2} There exists a strictly positive function $\eta=\eta(\e)$ and a positive constant $c_2\leqslant 1$ independent of $\e$ and $j$ such that
    \begin{equation}
    2c_2\eta(\e)\leqslant a_j^-(\e)+a_j^+(\e)\leqslant  2\eta(\e).\label{1.22}
    \end{equation}
\end{enumerate}

Our first main result concerns the resolvent convergence of $\mathcal{H}_\e$. We consider two main cases relating to various possible homogenized operators. In the first case we assume that
\begin{equation}\label{1.5}
\frac{1}{\e\ln\eta(\e)}\to -K,\quad\e\to+0,\quad K=\mathrm{const}\geqslant 0.
\end{equation}
We shall show that in this case the homogenized boundary condition on $\G_-$ is the Robin condition with certain coefficients. In order to formulate precisely this result we first introduce additional notations.

Given a constant $\mu\geqslant 0$, by $\mathcal{H}_{\mathrm{R}}^{(\mu)}$ we indicate the self-adjoint operator associated with the symmetric lower-semibounded closed sesquilinear form
%\begin{equation}\l%abel{1.6}
\begin{align*}
\mathfrak{h}_{\mathrm{R}}^{(\mu)}[u,v]:=&\big((\iu\nabla+A)u,(\iu\nabla+A) v\big)_{L_2(\Om)} +(Vu,v)_{L_2(\Om)}
\\
&+(bu,v)_{L_2(\Gamma_-)}+((K+\mu)\vt_\e' u,v)_{L_2(\Gamma_-)}\quad
\text{on}\quad \Ho^1(\Om,\G_+).
\end{align*}
%\end{equation}
In the same way as in \cite[Sec. 3]{IEOP} one can check that the domain of $\mathcal{H}_{\mathrm{R}}^{(\mu)}$ consists of the functions in $\H^2(\Om)$ satisfying the boundary conditions
\begin{equation}\label{1.7}
u=0\quad\text{on}\quad \G_+,\qquad \left(-\frac{\p}{\p x_2}-\iu A_2+b+(K+\mu)\vt_\e'\right)u=0\quad\text{on}\quad \G_-,
\end{equation}
while the action of the operator $\mathcal{H}_{\mathrm{R}}^{(\mu)}$ is described by the same differential expression as in (\ref{1.2}).

By $\|\cdot\|_{Z_1\to Z_2}$ we indicate the norm of an operator acting from a Banach space $Z_1$ to a Banach space $Z_2$.

Now we are ready to formulate our first main result.

\begin{theorem}\label{th1.1}
Suppose \ref{asA1}, \ref{asA2}, and (\ref{1.5}). Then the inequalities
\begin{align}
&\|(\mathcal{H}_\e+\iu)^{-1}-(\mathcal{H}_{\mathrm{R}}^{(\mu)}+\iu)^{-1}\|_{L_2(\Om)\to L_2(\Om)}\leqslant C \e(K+\mu)|\ln(K+\mu)\e|, \label{1.8}
\\
&\|(\mathcal{H}_\e+\iu)^{-1}-(\mathcal{H}_{\mathrm{R}}^{(0)}+\iu)^{-1}\|_{L_2(\Om)\to L_2(\Om)}\leqslant C (\e(K+\mu)|\ln(K+\mu)\e|+\mu), \label{1.9}
\\
&\|(\mathcal{H}_\e+\iu)^{-1}-(\mathcal{H}_{\mathrm{R}}^{(0)}+\iu)^{-1}\|_{L_2(\Om)\to \H^1(\Om)}\leqslant C (K+\mu)^{1/2}, \label{1.13}
\end{align}
hold true, where the constants $C$ are independent of $\e$, and $\mu=\mu(\e)$ is given by the formula
\begin{equation}\label{1.11}
\mu=\mu(\e):=-\frac{1}{\e\ln\eta(\e)}-K,\quad \mu(\e)\xrightarrow[\e\to+0]{}+0.
\end{equation}
There exists a boundary corrector $W=W(x,\e,\mu)$ defined explicitly in (\ref{3.47}) such that the inequality
\begin{equation}\label{1.12}
\|(\mathcal{H}_\e+\iu)^{-1}- (1+W)(\mathcal{H}_{\mathrm{R}}^{(\mu)}+\iu)^{-1}\|_{L_2(\Om)\to \H^1(\Om)} \leqslant C \e(K+\mu)|\ln(K+\mu)\e|,
\end{equation}
holds true, where the constant $C$ is independent of $\e$.
\end{theorem}

Let us discuss the results of this theorem. First of all we note that the operator $\mathcal{H}_\mathrm{R}^{(\mu)}$ depends on $\e$ not only via $\mu$, but also due to the presence of the boundary term $(K+\mu)\vt_\e'$ in (\ref{1.7}). However, we can easily get rid of such dependence, if we additionally assume that
%\begin{equation}\l%abel{1.122}
%\exists\ \vt_*:\mathds{R}\to \mathds{R}\text{ bounded such that }\sup\limits_{\mathds{R}}|\vt_\e'-\vt_*|\to0.
%\end{equation}
\begin{equation}\label{1.122}
\text{there exists a bounded function}\  \vt_*:\mathds{R}\to \mathds{R}
\ \text{such that }\sup\limits_{\mathds{R}}|\vt_\e'-\vt_*|\to0.
\end{equation}
Then we can replace the mentioned term by $(K+\mu)\vt_*$, and it will also generate additional term $C(K+\mu)\sup\limits_{\mathds{R}}|\vt_\e'-\vt_*|$  in the right hand sides of the estimates in Theorem~\ref{th1.1}.

We track the dependence of %D%
$K$ in the right hand sides in the estimates in Theorem~\ref{th1.1} just to involve also the case $K=0$, where, as we see, the rates of the convergence are better. We observe that the estimate (\ref{1.13}) is of interest only for $K=0$, while otherwise it states no convergence. Moreover, as $K\neq0$, there is no even strong resolvent convergence of $\mathcal{H}_\e$ to $\mathcal{H}_\mathrm{R}^{(0)}$ in the sense of the norm $\|\cdot\|_{L_2(\Om)\to\H^1(\Om)}$. Indeed, assume for simplicity that $\vt_\e$ is independent of $\e$ and for a given $f\in L_2(\Om)$ let $u_\e:=(\mathcal{H}_\e+\iu)^{-1}f$, $u_0:=(\mathcal{H}^{0}_\mathrm{R}+\iu)^{-1}f$. Then by the definition we have
\begin{align}
&\|(\iu\nabla+A)u_\e\|_{L_2(\Om)}^2+(Vu_\e,u_\e)_{L_2(\Om)}+(b u_\e,u_\e)_{L_2(\G_\e)}-\iu\|u_\e\|_{L_2(\Om)}^2=(f,u_\e)_{L_2(\Om)},
\nonumber
\\
&
\begin{aligned}
\|(\iu\nabla+A)u_0\|_{L_2(\Om)}^2+(Vu_0,u_0)_{L_2(\Om)}&+\big((b+K\vt') u_0,u_0\big)_{L_2(\G_-)}
\\
&-\iu\|u_0\|_{L_2(\Om)}^2=(f,u_0)_{L_2(\Om)}.
\end{aligned}
\label{1.366}
\end{align}
If we suppose the strong resolvent convergence at least for a sequence $\e\to+0$, we can pass to the limit in the former identity and we get
\begin{equation*}
\|(\iu\nabla+A)u_0\|_{L_2(\Om)}^2+(Vu_0,u_0)_{L_2(\Om)}+(b u_0,u_0)_{L_2(\G_\e)}-\iu\|u_0\|_{L_2(\Om)}^2=(f,u_0)_{L_2(\Om)}.
\end{equation*}
This contradicts to (\ref{1.366}). Nevertheless, due to (\ref{1.9}) the uniform resolvent convergence holds in the sense of the norm $\|\cdot\|_{L_2(\Om)\to L_2(\Om)}$ no matter whether $K$ vanishes or not.

We can see also from Theorem~\ref{th1.1} that the estimates for the rates of convergence depend highly on the operator norm and on the approximating operators. The worst estimate is in (\ref{1.13}). The first way to improve the convergence is to replace the norm by a weaker one as it was done in (\ref{1.8}) and (\ref{1.9}). The former estimate is better than the latter, but the price to pay is the dependence of the approximating operator $\mathcal{H}_\mathrm{R}^{(\mu)}$ on $\mu$. At the same time, this dependence is quite simple and it is %D%
only in the boundary condition (\ref{1.7}). If we keep the operator norm $\|\cdot\|_{L_2(\Om)\to\H^1(\Om)}$, we have to employ a special boundary corrector $W$ in order to get the reasonable estimate, see (\ref{1.12}). Here the rate of the convergence is the best one among given, but the approximating operator is the most complicated. We also note that these effects for pure Laplacian and strictly periodic alternation were already found in \cite{BBC1} in the %D%
particular case $K=b=0$.

In the second case we have $K=\infty$, or, more precisely,
\begin{equation}\label{1.24}
\e\ln\eta(\e)\to-0,\quad\e\to+0,
\end{equation}
and then the homogenized boundary condition on $\G_-$ is the Dirichlet one. In this case the homogenized operator is denoted as $\mathcal{H}_{\mathrm{D}}$ and is introduced as the self-adjoint one associated with the symmetric lower-semibounded closed sesquilinear form
%\begin{equation}\l%abel{1.25}
\begin{align*}
\mathfrak{h}_{\mathrm{D}}[u,v]:=&\big((\iu\nabla+A)u,(\iu\nabla+A) v\big)_{L_2(\Om)} +(Vu,v)_{L_2(\Om)}
\quad
\text{on}\quad \Ho^1(\Om,\G_-).
\end{align*}
%\end{equation}
Again by analogy with  \cite[Sec. 3]{IEOP} one can check that the domain of $\mathcal{H}_{\mathrm{D}}$ consists of the functions in $\H^2(\Om)$ vanishing on $\p\Om$, while the action is given by the differential expression in (\ref{1.7}).

In this case we replace the assumption \ref{asA2} by a stronger one.
\begin{enumerate}\def\theenumi{(A\arabic{enumi})} \setcounter{enumi}{2}
\item\label{asA3} There exists a strictly positive function $\eta=\eta(\e)$ and a positive constant $c_3$ independent of $\e$ and $j$ such that
    \begin{equation*}
    \eta(\e)\leqslant a_j^\pm(\e)\leqslant  c_3\eta(\e)\leqslant \frac{\pi}{4}.%\l%abel{1.26}
    \end{equation*}
\end{enumerate}

The result on the uniform resolvent convergence is formulated in
\begin{theorem}\label{th1.4}
Suppose \ref{asA1}, \ref{asA3}, and (\ref{1.24}). Then the inequality
\begin{equation*}
\|(\mathcal{H}_\e+\iu)^{-1}-(\mathcal{H}_{\mathrm{D}}+\iu)^{-1}\|_{L_2(\Om)\to \H^1(\Om)}\leqslant C\e^{1/4} \big(|\ln\sin\eta(\e)|+\cos\eta(\e)\big)^{1/4}, %\l%abel{1.27}
\end{equation*}
holds true, where the constant  $C$ is  independent of $\e$
\end{theorem}

With respect to Theorem~\ref{th1.1}, in Theorem \ref{th1.4} the estimate for the rate of convergence is given in the best possible norm and no corrector is involved. Here the rate of convergence can be improved only by assuming $\eta(\e)\to\pi/2-0$, since in this case the function $|\ln\sin\eta(\e)|+\cos\eta(\e)$ tends to zero. This is quite natural, since the assumption $\eta=\pi/2$ means $\G_\e=\emptyset$ and $\g_\e=\G_-$. Then there is no alternation of boundary conditions in this case and $\mathcal{H}_\e=\mathcal{H}_\mathrm{D}$.

The proven theorems and \cite[Ch.
VIII, Sec. 7, Ths. VIII.23, VIII.24]{RS1} imply the convergence for the spectrum and the spectral projectors of $\mathcal{H}_\e$. By $\spec(\cdot)$ we denote the spectrum of an operator.

\begin{theorem}\label{th1.2}
The spectrum of $\mathcal{H}_\e$ converges to that of $\mathcal{H}_0$, where
\begin{equation*}
\begin{aligned}
&\mathcal{H}_0:=\mathcal{H}_{\mathrm{R}}^{(0)}, \text{if we assume } \ref{asA1}, \ref{asA2}, (\ref{1.5}), (\ref{1.122}), \text{ and }(\ref{1.7}) \text{ holds for } \vt_* \text{ in place of } \vt'_\e,\\
&or \\
&\mathcal{H}_0:=\mathcal{H}_{\mathrm{D}}, \text{ if we assume }\ref{asA1}, \ref{asA3}, \text{ and }(\ref{1.24}).
\end{aligned}
\end{equation*}
Namely, if $\l\not\in\spec(\mathcal{H}_0)$, then
$\l\not\in\spec(\mathcal{H}_\e)$ for $\e$ small enough. If
$\l\in\spec(\mathcal{H}_0)$, then there exists
$\l_\e\in\spec(\mathcal{H}_\e)$ such that $\l_\e\to\l$ as $\e\to+0$.
The convergence of the spectral projectors associated with
$\mathcal{H}_\e$ and $\mathcal{H}_0$
\begin{equation*}
\|\mathcal{P}_{(a,b)}(\mathcal{H}_\e)-
\mathcal{P}_{(a,b)}(\mathcal{H}_0)\|_{L_2(\Om)\to L_2(\Om)}\to0,\quad \e\to0,
\end{equation*}
is valid for $a,b\not\in\spec(\mathcal{H}_0)$, $a<b$.
\end{theorem}

The next part of our results concerns the case of the pure Laplacian with the periodic alternation, namely, %D% the case
\begin{equation}\label{1.36}
A\equiv 0,\quad V\equiv 0,\quad b=\mathrm{const},\quad \mathcal{H}_\e=-\D, \quad \vt_\e(s) \equiv s, \quad a_j^\pm(\e)=\eta(\e).
\end{equation}
In this case the operator $\mathcal{H}_\e$ is periodic due to the periodicity of the boundary conditions. We apply the Floquet-Bloch expansion to characterize its spectrum. %D%Namely, w
We consider the cell of periodicity $\Om_\e:=\{x: |x_1|<\e\pi/2,\ 0<x_2<\pi\}$, and let $\gp_\e:=\p\Om_\e\cap\g_\e$, $\Gp_\e:=\p\Om_\e\cap\G_\e$, $\Gp_\pm:=\p\Om_\e\cap\G_\pm$. By $\Hpe(\tau)$ we denote the self-adjoint operator in $L_2(\Om_\e)$ associated with the lower-semibounded closed symmetric sesquilinear form
\begin{equation*}
\hpe(\tau)[u,v]:=\left( \left(\iu\frac{\p}{\p
x_1}-\frac{\tau}{\e}\right)u,\left(\iu\frac{\p}{\p
x_1}-\frac{\tau}{\e}\right)v\right)_{L_2(\Om_\e)}+ \left(\frac{\p
u}{\p x_2},\frac{\p v}{\p x_2}\right)_{L_2(\Om_\e)}+b(u,v)_{L_2(\Gp_\e)}
\end{equation*}
on $\Hoper^1(\Om_\e,\Gp_+\cup\gp_\e)$, where $\tau\in[-1,1)$ is the (rescaled) quasimomentum. The symbol
$\Hoper^1(\Om_\e,\Gp_+\cup\gp_\e)$ indicates the set of the functions in
$\Ho^1(\Om_\e,\Gp_+\cup\gp_\e)$ satisfying periodic boundary
conditions on the lateral boundaries of $\Om_\e$. The operator
$\Hpe(\tau)$ has a compact resolvent and its spectrum consists of countably many discrete eigenvalues. We denote them as
$\l_n(\tau,\e)$ and arrange in the ascending order counting the multiplicities. By \cite[Lm. 4.1]{BC} we know that
\begin{equation*}
\spec(\mathcal{H}_\e)=\essspec(\mathcal{H}_\e)=\bigcup\limits_{n=1}^\infty
\{\l_n(\tau,\e): \tau\in[-1,1)\},
\end{equation*}
where  $\essspec(\cdot)$ indicates  the essential spectrum of an operator.

Let $\mathfrak{L}_\e$ be the subspace of $L_2(\Om_\e)$
consisting of the functions independent of $x_1$. We decompose the space $L_2(\Om_\e)$ as
\begin{equation*}
L_2(\Om_\e)=\mathfrak{L}_\e\oplus \mathfrak{L}_\e^\bot,
\end{equation*}
where $\mathfrak{L}_\e^\bot$ is the orthogonal complement to
$\mathfrak{L}_\e$ in $L_2(\Om_\e)$. By $\mathcal{Q}_\mu$ we denote the
self-adjoint operator in $\mathfrak{L}_\e$ associated
with the lower-semibounded symmetric closed sesquilinear form
\begin{equation*}
\mathfrak{q}_\mu[u,v]:=\left(\frac{d u}{dx_2},\frac{d
v}{dx_2}\right)_{L_2(0,\pi)}+(b+K+\mu) u(0)\overline{v(0)} \quad
\text{on}\quad \Ho^1((0,\pi),\{\pi\}).
\end{equation*}
An alternative definition of $\mathcal{Q}_\mu$ is as the operator $-\frac{d^2}{dx_2^2}$ in $L_2(0,\pi)$, whose domain consists of the functions in
$\H^2(0,\pi)$ satisfying the boundary conditions
\begin{equation*}
u(\pi)=0,\quad u'(0)-(b+K+\mu)u(0)=0.
\end{equation*}

The uniform resolvent convergence for $\Hpe(\tau)$ and the asymptotic behavior of the band functions $\l_n(\tau,\e)$ read as follows.

\begin{theorem}\label{th1.3}
Suppose \ref{asA1}, \ref{asA2}, (\ref{1.5}), and (\ref{1.36}). Let $|\tau|<1-\k$, where $0<\k<1$ is a fixed constant. Then for sufficiently small $\e$ the estimate
\begin{equation}\label{1.18}
\left\| \left(\Hpe(\tau)-\frac{\tau^2}{\e^2}\right)^{-1} -
\mathcal{Q}_\mu^{-1}\oplus 0\right\|_{L_2(\Om_\e)\to
L_2(\Om_\e)}\leqslant C\e^{1/2}(K+\mu)\big(\k^{-1/2}+|\ln\e(K+\mu)|\big),
\end{equation}
holds true, where $\mu=\mu(\e)$ is defined by (\ref{1.11}), and $C$ is a constant independent of $\e$. Given any $N$, for $\e<2\k^{1/2}N^{-1}$ the eigenvalues $\l_n(\tau,\e)$, $n=1,\ldots,N$, satisfy the asymptotics
\begin{equation}\label{1.15}
\begin{aligned}
&\l_n(\tau,\e)=\frac{\tau^2}{\e^2}+\L_n(\mu)+R_n(\tau,\e,\mu),
\\
&|R_n(\tau,\e,\mu)|\leqslant Cn^4\e^{1/2}(K+\mu)\big(\k^{-1/2}+|\ln\e(K+\mu)|\big),
\end{aligned}
\end{equation}
where $\L_n(\mu)$, $n=1,\ldots,N$, are the first $N$ eigenvalues of $\mathcal{Q}_\mu$, and the constant $C$ is the same as in
(\ref{1.18}). The eigenvalues $\L_n(\mu)$ are holomorphic w.r.t. $\mu$ and solve the equation
\begin{equation}\label{1.17}
\sqrt{\L} \cos\sqrt{\L}\pi+(K+b+\mu)\sin\sqrt{\L}\pi=0.
\end{equation}
\end{theorem}

The proven theorem implies immediately

\begin{corollary}\label{th1.7}
Suppose \ref{asA1}, \ref{asA2}, (\ref{1.5}), and (\ref{1.36}). Then the length of the first band in the essential spectrum of $\mathcal{H}_\e$ is at least of order $\Odr(\e^{-2})$.
\end{corollary}

We mention that similar results for the homogenized Dirichlet and Neumann condition were established %D%
 in \cite[Sec.2]{BC}, \cite[Sec.2]{BBC1}.

In addition to the asymptotics of the first band functions, we describe the asymptotics for the bottom of the essential spectrum of $\mathcal{H}_\e$. We let
\begin{equation}\label{1.28}
\begin{aligned}
\tht(t_1,t_2):=&
2t_1\sum\limits_{n=1}^{\infty} \frac{1}{n \big( \sqrt{4n^2-t_2}+2n+t_1\big)}
\\
&-(t_2+t_1^2) \sum\limits_{n=1}^{\infty} \frac{1}{n \big( \sqrt{4n^2-t_2}+t_1\big)\big( \sqrt{4n^2-t_2}+2n+t_1\big)}
\end{aligned}
\end{equation}
We shall prove in Lemma~\ref{lm6.2} that the function $\tht$ is jointly holomorphic in $t_1$ and $t_2$.

\begin{theorem}\label{th1.5}
Suppose \ref{asA1}, \ref{asA2}, (\ref{1.5}), and (\ref{1.36}). Then the identity
\begin{equation}\label{1.30}
\inf\limits_{\tau\in[-1,1)} \l_1(\tau,\e)=\l_1(0,\e)
\end{equation}
holds true. The bottom $\l_1(0,\e)$ of the essential spectrum $\essspec(\mathcal{H}_\e)$ has the asymptotic expansion
\begin{equation}\label{1.34}
\l_1(0,\e)=\L\big(\e,\mu(\e)\big)+\Odr\Big((K+\mu)\E^{-2\e^{-1}} +\eta^{1/2}(K+\mu)+\e^{1/2}\eta^{1/2}(K+\mu)^{1/2}\Big),
\end{equation}
where $\L(\e,\mu)$ is the root of the equation
\begin{equation}\label{1.31}
\begin{aligned}
\big(\sqrt{\L}\cos\sqrt{\L}\pi+(b+K+\mu)\sin\sqrt{\L}\pi
&\big)\big(1-\e(K+\mu)\tht(\e b,\e^2\L)\big)
\\
&+\e(K+\mu)^2\tht(\e b,\e^2\L)\sin\sqrt{\L}\pi=0,
\end{aligned}
\end{equation}
satisfying
\begin{equation}\label{1.35}
\L(\e,\mu)=\L_1(\e,\mu)+o(1),\quad \e\to+0,\quad\mu\to+0.
\end{equation}
The function $\L$ is jointly holomorphic in $\e$ and $\mu$ and can be represented as the uniformly convergent series
\begin{equation}\label{1.32}
\L(\e,\mu)=\L_1(\mu)+\sum\limits_{j=2}^{\infty}\e^j \Ups_j(\mu),
\end{equation}
where the functions $\Ups_j$ are holomorphic in $\mu$. All these functions can be found explicitly and, in particular, we have
\begin{equation}\label{1.33}
\begin{aligned}
&\Ups_1(\mu)=\frac{\pi^2 b}{6} \frac{\L_1(\mu)\big(\L_1(\mu)+\mu\big)^2} {\pi (K+b+\mu)^2+\L_1\pi+b+K+\mu},
\\
&\Ups_2(\mu)=-\frac{\z(3)}{4} \frac{\L_1(\mu)(K_0+\mu)^2\big(\L_1(\mu)+2b^2
\big)}{\pi (K+b+\mu)^2+\L_1\pi+b+K+\mu},
\end{aligned}
\end{equation}
where $\z(t)$ is the Riemann zeta-function.
\end{theorem}

The first result of this theorem, namely the identity (\ref{1.30}), was proven in \cite[Sec.2]{BBC1} and \cite[Sec.2]{BC} for the case of the homogenized Dirichlet and Neumann condition. The main result of Theorem~\ref{th1.5} is the asymptotics (\ref{1.34}). Although it contains just one term $\L(\e,\mu)$, this term is holomorphic w.r.t. $\e$ and $\mu$ and is defined quite explicitly as the solution to the equation (\ref{1.31}), (\ref{1.35}). This solution can be represented as the convergent series (\ref{1.32}) and the latter can be regarded as the asymptotics for $\L$. In other words, the power in $\e$ and $\eta$ part in the asymptotics for $\L$ can be summed up to the exponentially small error, see (\ref{1.34}). It is a strong result for the problem in the homogenization theory, since usually complete asymptotic expansions in the homogenization theory can not be summed up. We also mention that previously similar result was only known in the case $K=b=0$, see \cite[Sec.2]{BBC1}.

In conclusion to this section let us discuss some open questions related to the studied model. The main question is the gap conjecture, namely, the question on the gaps in the essential spectrum of $\mathcal{H}_\e$. It is not known whether such gaps exist or not. If exist, it would be interesting to prove the Bethe-Sommerfeld conjecture for our model, namely, that the number of the gaps is finite. The next question would be the dependence of the number, the location, and the sizes of the gaps on $\e$. These questions were formulated even in the first paper \cite{BC}, where the studied model was proposed. Our conjecture was that the gaps exist and a possible way of proving it was to study in more details the behavior of the band functions as $\e\to+0$. In fact, the idea was to extend the theorems like Theorem~\ref{th1.3}, trying to include into consideration vicinities of the points $\tau=\pm -1$, to construct the complete asymptotic expansions for $\l_n$, and to get these results not only for the first ones but for all band functions $\l_n$. Now we have to admit that these questions are much more complicated as we expected in the beginning. In particular, one needs to develop a new technique in comparison with ours to construct the complete asymptotic expansions and/or to include the vicinities of the points $\tau=\pm 1$. Moreover, there is a usual problem with ordering of the band functions. Indeed, even in the case $b=K=0$, if we consider the homogenized operator as the periodic one with the period $\e\pi$, then the associated band functions are $\frac{(2m+\tau)^2}{\e^2}+\left(n+\frac{1}{2}\right)^2,\  m,n\in \mathds{Z}_+$. Then the issue of putting these eigenvalues in the ascending order for all values of $\e$ is not trivial. Attempts to improve (\ref{1.18}) seem to be not useful, since such convergence can not describe in an appropriate way the desired asymptotic behavior of the band functions. This is why to solve the gap conjecture for our model, one has to develop a new  approach. It is likely that the approaches used in the proofs of the Bethe-Sommerfeld conjecture for various operators can be useful. Moreover, now we conjecture that our model has no gaps in the essential spectrum. The first argument supporting such conjecture is that the homogenized operator has no gaps at all and the number of overlapping bands in its essential spectrum increases as the point goes to infinity. The second argument is that we can make the rescaling in the original operator $x\mapsto x\e^{-1}$. Under such rescaling the operator $\mathcal{H}_\e$ becomes the operator $-\D$ in the strip $\{x: 0<x_2<\pi\e^{-1}\}$ with fixed alternation of the boundary conditions. As $\e\to+0$, the new strip ``tends'' to the half-space $x_2>0$. Hence, one can expect that the spectrum of the rescaled operator converges in certain sense to that of $-\D$ in $x_2>0$ with the original fixed alternation. And since the spectrum of the latter is always $[0,+\infty)$, we can conjecture that there is no gaps in the perturbed spectrum. If the essential spectrum indeed has no gaps, then Theorem~\ref{th1.5} and its analog in \cite{BC} gives the complete description of the location of the essential spectrum, namely, of its bottom.

\section{Resolvent convergence: homogenized Robin condition}

In this section we prove Theorem~\ref{th1.1}.

Given a function $f\in L_2(\Om)$, we denote
\begin{equation*}
u_\e:=(\mathcal{H}_\e-\iu)^{-1}f,\quad
u^{(\mu)}:=(\mathcal{H}_{\mathrm{R}}^{(\mu)}-\iu)^{-1}f.
\end{equation*}

We first prove an auxiliary lemma which will be one of the key ingredients in the proof of Theorem~\ref{th1.1}.

\begin{lemma}\label{lm3.1}
Let $W=W(x,\e,\mu)$ be a real function
belonging to
\begin{equation}\label{3.0}
 C(\overline{\Om})\cap
C^2(\overline{\Om}\setminus\{x: x_1=\e s_j^\e\pm \e a_j^\pm(\e),\  x_2=0,\ j\in\mathds{Z}\}),
\end{equation}
satisfying boundary conditions
\begin{equation}\label{3.1}
W=-1 \quad\text{on}\quad \g_\e,\qquad \frac{\p W}{\p
x_2}=-(K+\mu)(1+\e\vp_\e)\vt_\e'\quad\text{on}\quad \G_\e,
\end{equation}
bounded uniformly in $\overline{\Om}$, and
having differentiable asymptotics
\begin{equation}\label{3.2}
W(x,\e,\mu)=c_j^\pm(\e,\mu)
\sqrt{r_j^\pm}\sin\frac{\tht_j^\pm}{2}+\Odr(r_j^\pm),\quad r_j^\pm\to+0.
\end{equation}
Here $\vp_\e=\vp_\e(x_1)\in C(\mathds{R})$ is a bounded function, $(r_j^\pm,\tht_j^\pm)$ are polar coordinates centered at $(\e s_j^\e \pm\e a_j^\pm(\e),0)$ such that the values $\tht_j^\pm=0$ correspond to the
points of $\g_\e$. Assume also that $\D W\in C(\overline{\Om})$ and the function $\D W$ is bounded uniformly in $\overline{\Om}$. Let
\begin{equation}\label{3.6}
v_\e:=u_\e-(1+W)u^{(\mu)}.
\end{equation}
Then $(1+W)u^{(\mu)}$ belongs to $\Ho^1(\Om,\G_+\cup\g_\e)$, and
\begin{equation}\label{3.3}
\begin{aligned}
\|(\nabla+\iu A) v_\e&\|_{L_2(\Om)}^2+(V v_\e, v_\e)_{L_2(\Om)}+ (b v_\e,v_\e)_{L_2(\G_\e)}
+\iu\|v_\e\|_{L_2(\Om)}^2
\\
=& (f,v_\e
W)_{L_2(\Om)}+(u^{(\mu)}\D W,v_\e)_{L_2(\Om)}-2\iu(u^{(\mu)}W, v_\e)_{L_2(\Om)}
\\
&-2(Vu^{(\mu)},Wv_\e)_{L_2(\Om)}-2\big(W(\nabla+\iu A)u^{(\mu)}, (\nabla+\iu A)v_\e\big)_{L_2(\Om)}
\\
&-\Big(\big(2b+(K+\mu)(1+\e\vp_\e)\vt_\e'\big)u^{(\mu)}, Wv_\e\Big)_{L_2(\G_\e)}.
\end{aligned}
\end{equation}
\end{lemma}

\begin{proof}
The functions $u_\e$ and $u^{(\mu)}$ satisfy the integral identities
\begin{equation}\label{3.4}
\begin{aligned}
\big((\nabla+\iu A)u_\e, & (\nabla+\iu A) \phi\big)_{L_2(\Om)}+(Vu_\e,\phi)_{L_2(\Om)}
\\
& + (bu_\e,\phi)_{L_2(\G_\e)}
+\iu (u_\e,\phi)_{L_2(\Om)}=
(f,\phi)_{L_2(\Om)}
\end{aligned}
\end{equation}
for all $\phi\in \Ho^1(\Om,\G_+\cup\g_\e)$, and
\begin{equation}
\begin{aligned}
\big((\nabla+\iu A)u^{(\mu)}, & (\nabla+\iu A) \phi\big)_{L_2(\Om)}+(Vu^{(\mu)},\phi)_{L_2(\Om)}
\\
& + \big((b+(K+\mu)\vt_\e')u^{(\mu)},\phi\big)_{L_2(\G_\e)}
+\iu (u^{(\mu)},\phi)_{L_2(\Om)}=
(f,\phi)_{L_2(\Om)}
\end{aligned} \label{3.5}
\end{equation}
for all $\phi\in \Ho^1(\Om,\G_+)$. By analogy with the proof of Lemma~3.2 in \cite{BC}, one can show that
$(1+W)\phi\in
\Ho^1(\Om,\G_+\cup\g_\e)$ for each $\phi$ which belongs to the domain of either
$\mathcal{H}_\e$ or $\mathcal{H}_{\mathrm{R}}^{(\mu)}$. Hence
\begin{equation*}%\l%abel{3.7}
(1+W)u^{(\mu)}\in
\Ho^1(\Om,\G_+\cup\g_\e),\quad (1+W)v_\e\in \Ho^1(\Om,\G_+\cup\g_\e).
\end{equation*}

We choose the test function in (\ref{3.5}) as $\phi=(1+W)v_\e$. It yields
\begin{align*}
\big((\nabla+\iu A) u^{(\mu)}&,(\nabla+\iu A) (1+W)v_\e)_{L_2(\Om)}+(V u^{(\mu)}, (1+W)v_\e)_{L_2(\Om)}
\\
& + \big((b+(K+\mu)\vt_\e') u^{(\mu)}, (1+W) v_\e\big)_{L_2(\G_-)}+\iu (u^{(\mu)},(1+W)v_\e)_{L_2(\Om)}
\\
= &(f,(1+W)v_\e)_{L_2(\Om)},
\end{align*}
and we rewrite this identity as
\begin{align*}
\big((\nabla+\iu A) u^{(\mu)}&,(1+W)(\nabla+\iu A) v_\e)_{L_2(\Om)}+(V u^{(\mu)}, (1+W)v_\e)_{L_2(\Om)}
\\
& + (b u^{(\mu)}, (1+W) v_\e)_{L_2(\G_-)}+\iu (u^{(\mu)},(1+W)v_\e)_{L_2(\Om)}
\\
= &(f,(1+W)v_\e)_{L_2(\Om)}-\big((\nabla+\iu A) u^{(\mu)},v_\e\nabla W)_{L_2(\Om)}
\\
&-
\big((K+\mu)\vt_\e' u^{(\mu)},(1+W)v_\e\big)_{L_2(\G_\e)},
\end{align*}
and
\begin{align*}
\big((\nabla+\iu A) (1+W)& u^{(\mu)},(\nabla+\iu A) v_\e)_{L_2(\Om)}+(V(1+W) u^{(\mu)},v_\e)_{L_2(\Om)}
\\
& + (b (1+W) u^{(\mu)}, v_\e)_{L_2(\G_-)}+\iu ((1+W)u^{(\mu)},v_\e)_{L_2(\Om)}
\\
= &(f,(1+W)v_\e)_{L_2(\Om)}-
\big((K+\mu)\vt_\e' u^{(\mu)},(1+W)v_\e\big)_{L_2(\G_\e)}
\\
&
-\big((\nabla+\iu A) u^{(\mu)},v_\e\nabla W)_{L_2(\Om)}
+\big( u^{(\mu)}\nabla W,(\nabla+\iu A)v_\e)_{L_2(\Om)}.
\end{align*}
We deduct the last identity from (\ref{3.4}) with $\phi=v_\e$ and get,
\begin{equation}\label{3.8}
\begin{aligned}
\|(\nabla+\iu A) v_\e&\|_{L_2(\Om)}^2+(V v_\e, v_\e)_{L_2(\Om)}+ (b v_\e,v_\e)_{L_2(\G_\e)}
+\iu\|v_\e\|_{L_2(\Om)}^2
\\
=&-(f,W v_\e)_{L_2(\Om)}+  \big((K+\mu)\vt_\e' u^{(\mu)},(1+W)v_\e\big)_{L_2(\G_\e)}
\\
&
+\big((\nabla+\iu A) u^{(\mu)},v_\e\nabla W)_{L_2(\Om)}
-\big( u^{(\mu)}\nabla W,(\nabla+\iu A)v_\e)_{L_2(\Om)}.
\end{aligned}
\end{equation}
Now we consider separately the two last terms in the right hand side of equation (3.9). We integrate by parts and employ the boundary condition (\ref{3.1}), the identity (\ref{3.5}), and the belongings (\ref{3.6}),
\begin{align*}
&\big((\nabla+\iu A) u^{(\mu)},v_\e\nabla W)_{L_2(\Om)}
-\big( u^{(\mu)}\nabla W,(\nabla+\iu A)v_\e)_{L_2(\Om)}
\\
&=\big((\nabla+\iu A) u^{(\mu)},v_\e\nabla W)_{L_2(\Om)}+ \int\limits_{\G_\e}
u^{(\mu)}\frac{\p W}{\p x_2}\overline{v}_\e\di x_1
\\
&\hphantom{=}\,+ (\mathrm{div}\,
u^{(\mu)}\nabla W,v_\e)_{L_2(\Om)}- (u^{(\mu)}\nabla W,\iu A v_\e)_{L_2(\Om)}
\\
&=2\big((\nabla+\iu A) u^{(\mu)},v_\e\nabla W)_{L_2(\Om)}+ (u^{(\mu)}\D W,v_\e)_{L_2(\Om)}
\\
&\hphantom{=}\,-\big((K+\mu) (1+\e\vp_\e)\vt_\e' u^{(\mu)}, v_\e\big)_{L_2(\G_\e)},
\end{align*}
and
\begin{align*}
\big((\nabla+\iu A)u^{(\mu)}, v_\e \nabla W\big)_{L_2(\Om)}= & \big((\nabla+\iu A)u^{(\mu)}, (\nabla+\iu A) v_\e W\big)_{L_2(\Om)}
\\
&- \big((\nabla +\iu A)u^{(\mu)}, W (\nabla +\iu A) v_\e\big)_{L_2(\Om)}
\\
=&(f,W v_\e)_{L_2(\Om)} -(V u^{(\mu)},W v_\e)_{L_2(\Om)} -\iu (u^{(\mu)}, W v_\e)_{L_2(\Om)}
\\
&-\big((b+(K+\mu)\vt_\e')u^{(\mu)}, Wv_\e\big)_{L_2(\G_\e)}
\\
&- \big((\nabla +\iu A)u^{(\mu)}, W (\nabla +\iu A) v_\e\big)_{L_2(\Om)}.
\end{align*}
These equations and (\ref{3.8}) imply (\ref{3.3}).
\end{proof}

Our next step is to construct the function $W$ satisfying the hypothesis of the proven lemma. Then we shall be able to estimate the right hand side of (\ref{3.3}) and as a result the
$\H^1(\Om)$-norm of $v_\e$. Exactly the last estimate will allow us to prove Theorem~\ref{th1.1}.

To construct the function $W$, we employ the technique from \cite{B03}, \cite{AHP-4}, \cite{AHP-25}, \cite{AHP-24}  based on the method of matching asymptotic expansions \cite{Il}, the boundary layer method \cite{VL}, and the multiscale method \cite{MSM}. For the periodic case similar approach has already been successfully applied in \cite[Sec.3]{BBC1}. In order to treat the nonperiodic case we employ some ideas in \cite{BBC1} in combination with the technique of \cite{B03}.

We seek $W$ as a formal asymptotic solution to the equation
\begin{equation*}%\l%abel{3.9}
\D W=0\quad\text{in}\quad\Om
\end{equation*}
satisfying the boundary conditions (\ref{3.1}). This solution is constructed as a sum of an outer expansion, a boundary layer, and an inner expansion. We begin with the boundary layer.

The main idea of introducing the boundary layer is to satisfy the Neumann condition on $\G_\e$. We first define the rescaled variables
\begin{equation}\label{3.34}
\xi=(\xi_1,\xi_2),\quad \xi_1=\frac{\vt_\e(x_1)}{\e},\quad \xi_2=\vt_\e'(x_1)\frac{x_2}{\e}.
\end{equation}
This change of variables maps the points $x=(\e s_j^\e,0)$ into a periodic set $\xi=(\pi j,0)$, while the Laplacian becomes $\e^{-2}\big(\vt_\e'(x_1)\big)^2\D_\xi+\Odr(\e^{-1})$. These two facts are exactly the motivation of defining the rescaled variables for the boundary layer by (\ref{3.34}). We construct the latter as
\begin{equation}\label{3.35}
W_\e^{\mathrm{bl}}(\xi,x_1,\mu)=\e (K+\mu)\big(1+\e\vp_\e(x_1)\big) X(\xi),
\end{equation}
and the function $\vp_\e$ will be determined below.

We substitute (\ref{3.35}) into the boundary value problem for $W$, replace then the set $\G_\e$ by $\G_-\setminus\bigcup\limits_{j\in\mathds{Z}} (\e s_j^\e,0)$, and equate the coefficients at the leading power of $\e$. It implies the boundary value problem for $X$,
\begin{equation}\label{3.36}
\D_\xi X=0,\quad \xi_2>0,\qquad \frac{\p X}{\p\xi_2}=-1, \quad
\xi\in\G^0:=\{\xi: \xi_2=0\}
\setminus\bigcup\limits_{j=-\infty}^{+\infty}(\e\pi j,0).
\end{equation}
We are interested in a solution to this problem which is $\e\pi$-periodic in $x_1$ and decays exponentially
 as $\xi_2\to+\infty$. Such solution was found explicitly in \cite{AHP-23},
\begin{equation}\label{3.37}
X(\xi):=\RE\ln\sin(\xi_1+\iu\xi_2)+\ln 2-\xi_2,
\end{equation}
where the branch of the logarithm is fixed by the requirement  $\ln 1=0$. Some additional properties of this function were established. Namely, it was shown that the function $X$ belongs to $C^\infty(\{\xi: \xi_2>0\}\cup\G^0)$
and satisfies the differentiable asymptotics
\begin{equation}\label{3.13}
X(\xi)=\ln|\xi-(\pi j,0)|+\ln 2-\xi_2+\Odr(|\xi-(\pi
j,0)|^2),\quad \xi\to(\pi j,0),\quad j\in\mathds{Z}.
\end{equation}

The constructed boundary layer does not satisfy the required boundary condition (\ref{3.1}) on $\g_\e$ and it has also logarithmic singularities at the points $(\e s_j^\e,0)$ because of (\ref{3.13}). This is the reason why in a vicinity of these points we introduce the inner expansion and construct it by using the method of matching asymptotic expansions. In a vicinity of each point $(\e s_j^\e,0)$ we introduce one more rescaled variables,
\begin{equation*}
\vs^{(j)}=(\vs^{(j)}_1,\vs^{(j)}_2), \quad \vs^{(j)}_1:= (\xi_1-\pi j)\eta^{-1},\quad \vs^{(j)}_2:=\xi_2\eta^{-1}.
\end{equation*}
It follows from (\ref{3.13}), (\ref{3.35}) that
\begin{equation}\label{3.39}
\begin{aligned}
W_\e^{\mathrm{bl}}(\xi,x_1,\mu)=&-1-\e\vp_\e(x_1)+\e (K+\mu)\big(1+\e\vp_\e(x_1)\big) \big(\ln|\vs^{(j)}|+\ln 2-\eta\vs_2^{(j)}\big)
\\
&+\Odr\big(\e(K+\mu)\eta^2|\vs^{(j)}|^2\big),\quad \xi\to(\pi j,0).
\end{aligned}
\end{equation}
Hence, in accordance with the method of matching the asymptotic expansions we should construct the inner expansion for $W$ as
\begin{equation}\label{3.38}
W_\e^{\mathrm{in},j}(\vs^{(j)},x_1,\mu)=-1+\e Y^{(j)}(\vs^{(j)},\e),
\end{equation}
and the functions $Y^{(j)}(\vs^{(j)},\e)$ should satisfy the asymptotics
\begin{equation}\label{3.41}
Y^{(j)}(\vs^{(j)},\e)=(K+\mu)\big(1+\e\vp_\e(x_1)\big)(\ln|\vs^{(j)}|+\ln 2)-\vp_\e(x_1)+o(1),\quad |\vs^{(j)}|\to\infty.
\end{equation}
We substitute (\ref{3.38}) into the boundary value problem for $W$ and equate the coefficients at the leading orders of the small parameters $\e$ and $\eta$. It leads us to the boundary value problem for $Y^{(j)}$,
\begin{equation}\label{3.40}
\begin{aligned}
&\D_{\vs^{(j)}} Y^{(j)}=0,\quad \vs_2^{(j)}>0,
\\
&Y^{(j)}=0,\quad \vs^{(j)}\in\g^{1,j}:=\{\vs^{(j)}:\, -2\a_j^-(\e)<\vs_1^{(j)}<2\a_j^+(\e),\,\vs^{(j)}_2=0\},
\\
&\frac{\p
Y^{(j)}}{\p\vs_2^{(j)}}=0,\quad \vs^{(j)}\in\G^{1,j}:= O\vs_1
\setminus\overline{\g^{1,j}},
\\
&\a_j^\pm(\e):=\pm\frac{\vt_\e(\e s_j^\e\pm\e a_j^\pm(\e))-\vt_\e(\e s_j^\e)}{2\e\eta}.
\end{aligned}
\end{equation}
We observe that due to (\ref{1.23}), (\ref{1.19}), (\ref{1.20}), (\ref{1.22}) the functions $\a_j^\pm(\e)$ satisfy the estimate
\begin{equation*}%\l%abel{3.44}
|\a_j^\pm(\e)|\leqslant C,
\end{equation*}
where the constant $C$ is independent of $\e$ and $j$.

In \cite{AHP-23} a special solution to the problem (\ref{3.40}) was found for the case $\a_j^\pm(\e)=\pm 1$ being
\begin{equation}\label{3.42}
Y(\vs):=\RE\ln (z+\sqrt{z^2-1}),\quad z=\vs_1+\iu \vs_2.
\end{equation}
Here the branch  of the   square root is fixed by the requirements  $\sqrt{1}=1$. It was shown that the function $Y$ belongs to $C^\infty(\{\vs: \vs_2\geqslant 0,\ \vs\not=(\pm 1,0)\})$ and satisfies the asymptotics
\begin{equation}\label{3.15}
Y(\vs)=\ln|\vs|+\ln 2+\Odr(|\vs|^{-2}),\quad \vs\to\infty.
\end{equation}
By a trivial change of variables we can obtain then the solution for our case,
\begin{equation*}%\l%abel{3.43}
Y^{(j)}(\vs^{(j)},\e)=(K+\mu)\big(1+\e\vp_\e(x_1)\big) Y\left(\frac{\vs^{(j)}_1+\a_j^-(\e)-\a_j^+(\e)
}{d_j(\e)}, \frac{\vs^{(j)}_2} {d_j(\e)}\right).
\end{equation*}
Hence, the function $Y^{(j)}$ has the same smoothness as $Y$ and
\begin{equation}\label{3.45}
\begin{aligned}
Y^{(j)}(\vs^{(j)},\e)=&(K+\mu)\big(1+\e\vp_\e(x_1)\big)\bigg(\ln|\vs^{(j)}|+\ln 2-\ln d_j(\e)
\\
&+ \big(\a_j^-(\e)-\a_j^+(\e)\big) \frac{\vs_1^{(j)}}{|\vs^{(j)}|^2}+\Odr(|\vs^{(j)}|^{-2})\bigg),\quad |\vs^{(j)}|\to+\infty.
\end{aligned}
\end{equation}
We compare this asymptotics with (\ref{3.41}) and see that the function $\vp_\e$ should satisfy the identity
\begin{equation}\label{3.46}
(K+\mu)\big(1+\e\vp_\e(x_1)\big) \ln d_j(\e)=\vp_\e(x_1)\quad \text{as}\quad x_1-\e s^\e_j=\Odr(\e\eta).
\end{equation}
We define such function as follows. Let $\chi$ be an infinitely differentiable cut-off function with values in $[0,1]$, being one as $t<1/4$ and vanishing as $t>3/4$. Denote
\begin{align*}
&d_j(\e):=\frac{\vt_\e\big(\e s_j^\e+\e a_j^+(\e)\big)-\vt_\e\big(\e s_j^\e-\e a_j^-(\e)\big)}{2\e\eta(\e)}
\end{align*}
and let
\begin{equation}\label{3.49}
g_\e(\tht):=d_{j+1}(\e)-\chi\bigg(\frac{\tht}{\e\pi}\bigg) \big(d_{j+1}(\e)-d_j(\e)\big),\quad \e\pi j\leqslant\tht\leqslant\e\pi(j+1),\quad j\in\mathbb{Z}.
\end{equation}
It is easy to see that in a $\e$-neighborhood of each point $\e s^\e_j$ we have $g_\e(\vt_\e(x_1))=d_j(\e)$. Then we replace $d_j$ in (\ref{3.46}) by $g_\e(\vt_\e(x_1))$ and determine $\vp_\e$,
\begin{equation}\label{3.50}
\vp_\e(x_1):=\frac{(K+\mu)\ln g_\e\big(\vt_\e(x_1)\big)}{1-\e(K+\mu)\ln g_\e(\vt_\e(x_1))}.
\end{equation}
It follows from the assumptions~\ref{asA1},~\ref{asA2} that the functions $g_\e(\vt_\e(x_1))$ and $\vt_\e(x_1)$ are bounded uniformly in $x_1$ and $\e$.

As we see from (\ref{3.39}), (\ref{3.45}), the term $-\e(K+\mu)\xi_2$ is not matched with any term in the inner expansion.
 It was found in \cite{AHP-4}, \cite{AHP-25}, \cite{AHP-24} that this term should be either matched or canceled in order
to have a good approximation by the constructed asymptotics. This is also the case for our problem and to compensate it, we add  the term $\e(K+\mu)\xi_2$ in the boundary layer and we introduce the term $-\e(K+\mu) x_2$ as the outer expansion. With this trick the final form for $W$ is
\begin{equation}\label{3.47}
\begin{aligned}
W(x,\e,\mu)=&-(K+\mu)\big(1+\e\vp_\e(x_1)\big) \vt_\e'(x_1) x_2 \\
&+ \e (K+\mu) \big(1+\e \vp_\e(x_1)\big) (X(\xi)+\xi_2) \prod\limits_{j=-\infty}^{+\infty} \big(1-\chi(|\vs^{(j)}|\eta^{3/4})\big)
\\
&+ \sum\limits_{j=-\infty}^{+\infty} \chi(|\vs^{(j)}|\eta^{3/4}) \big(-1+\e Y^{(j)}(\vs^{(j)},\e)\big).
\end{aligned}
\end{equation}
This function is well-defined since the product and the sum in its definition are always finite.

Let us check that the function $W$ satisfies the assumption of Lemma~\ref{lm3.1}. By direct calculations we check that the belonging (\ref{3.0}), the boundary conditions (\ref{3.1}), and the asymptotics (\ref{3.2}) hold. It follows from (\ref{3.0}), (\ref{3.1}) that $\D W\in C(\overline{\Om})$. It is also easy to check that this function is bounded uniformly in $\overline{\Om}$. Therefore, we can apply Lemma~\ref{lm3.1}.

Our next step is to estimate the right hand side of (\ref{3.3}). In order to do it, we need several auxiliary lemmas. The first two of them were proven in \cite{BBC1} for the periodic case $\vt_\e(s)\equiv s$, $a_j^\pm(\e)\equiv 1$. For our non-periodic case the proof is the same with several obvious minor changes. This is why we provide these lemmas without proofs.

Denote
\begin{equation*}
\Om^\d:=\bigcup\limits_{j=-\infty}^{+\infty}  \{x: |\xi- (s_j^\e,0)|< \d\}\cap\Om, \quad \g_\e^\d:= \bigcup\limits_{j=-\infty}^{+\infty} \{x: |\xi_1-s_j^\e|<\d, \, x_2=0\}.
\end{equation*}

\begin{lemma}\label{lm3.3}
For any $\d\in(0,\pi/2)$, $u\in\H^1(\Om)$, $v\in\H^2(\Om)$  the inequalities
\begin{align*}
&\|u\|_{L_2(\Om^\d)}\leqslant C\d \big(|\ln\d|^{1/2}+1\big)
\|u\|_{\H^1(\Om)}, %\l%abel{2.21}
\\
&\|v\|_{L_2(\g_\e^\d)}\leqslant C\d^{1/2}\|v\|_{\H^2(\Om)}, %\l%abel{3.48}
\end{align*}
hold true, where the constants $C$ are independent of $\e$, $\d$, $u$, and $v$.
\end{lemma}

\begin{lemma}\label{lm3.2}
For $\d\in[3\eta^{1/4}/4,\pi/2)$ the estimates
\begin{align}
|W|\leqslant C\e(K+\mu)(|\ln\d|+1), \quad x\in\Om\setminus\Om^\d, \quad
|W|\leqslant C, \quad x\in\Om^\d, \label{3.19}
\end{align}
are valid, where the constants $C$ are independent of $\e$, $\mu$, $\eta$, $\d$, and $x$.
\end{lemma}
Due to the non-periodic structure of the alternation of boundary conditions, we need to prove the next two lemmas. Denote
\begin{equation*}
\Om_{\e,j}:=\left\{x: |\xi_1 -\pi j|<\frac{\pi}{2}\right\}.
\end{equation*}

\begin{lemma}\label{lm3.4}
The estimate
\begin{equation*}
\|\D W\|_{L_2(\Om_{\e,j})}\leqslant C\e(K+\mu)
\end{equation*}
holds true, where the constant $C$ is independent of $\e$  and $j$.
\end{lemma}

\begin{proof}
Throughout the proof by $C$ we denote inessential constants independent of $\e$ and $j$. We first calculate $\D W$ employing the equations in (\ref{3.36}), (\ref{3.40}) for $X$ and $Y^{(j)}$,
\begin{align*}
&\D W=-(K+\mu) x_2 \big((1+\e\vp_\e)\vt_\e'\big)''
\\
&\hphantom{\D W=}+ \e(K+\mu) \left(2\vp_\e'\vt_\e' \frac{\p X}{\p\xi_2}+\e \vp''_\e (X+\xi_2)\right) \prod\limits_{j=-\infty}^{+\infty} \big(1-\chi(|\vs^{(j)}|\eta^{3/4})\big)
\\
&\hphantom{\D W=}+2 \sum\limits_{j=-\infty}^{+\infty} \nabla \chi(|\vs^{(j)}|\eta^{3/4})\cdot\nabla  W_j^{\mathrm{mat}}+ \sum\limits_{j=-\infty}^{+\infty} W_j^{\mathrm{mat}}\D \chi(|\vs^{(j)}|\eta^{3/4}),
\\
&W_j^{\mathrm{mat}}(x,\e,\mu)=-1+\e Y^{(j)}(\vs^{(j)},\e)-\e (K+\mu) \big(1+\e\vp_\e(x_1)\big) (X(\xi)+\xi_2),
\end{align*}
where the arguments of $Y$ are $\Big((\vs^{(j)}_1+\a_j^-(\e)-\a_j^+(\e)
)d_j^{-1}(\e),  \vs^{(j)}_2 d_j^{-1}(\e)\Big)$.

For $x\in\Om^{\frac{3}{4}\eta^{1/4}}$ the most part of the terms in $\D W$ vanishes and we have
\begin{equation*}
\D W=\e(K+\mu) \left(2\vp_\e'\vt_\e' \frac{\p X}{\p\xi_2}+\e \vp''_\e  X\right),\quad x\in\Om^\d.
\end{equation*}
It follows from (\ref{1.20}), (\ref{1.21}) that
\begin{equation}\label{3.51}
|\vp_\e'|\leqslant C\e^{-1},\quad |\vp_\e''|\leqslant C\e^{-2},
\end{equation}
where the constant $C$ is independent of $\e$ and $x_1$. Moreover, since by (\ref{3.49}), (\ref{3.50}) the function $g_\e(\vt_\e)$ %D%, $\vp_\e$
is constant as $|\xi_1-\pi j|<\pi/4$,
the derivatives of $\vp_\e$ vanish for such values of $x_1$. Thus, for some small fixed $\d\in(0,\pi/2)$,
\begin{equation}\label{3.49a}
\begin{aligned}
\|\D W\|_{L_2\big(\Om_{\e,j}\cap \Om^{\frac{3}{4}\eta^{1/4}}\big)}^2&\leqslant C\e^2(K+\mu)^2 \|X\|_{\H^1(\{\xi: \xi_2>0,\ \d<|\xi_1|<\pi/2\})}
\\
&\leqslant C\e^2(K+\mu)^2,
\end{aligned}
\end{equation}
where the constant $C$ is independent of $\e$ and $j$.

As $x\in \Om_{\e,j}\setminus \Om^{\frac{1}{4}\eta^{1/4}}$, the functions $\chi(|\z^{(j)}|\eta^{3/4})$ are identically one and thus
\begin{equation*}
\D W=-(K+\mu)x_2 \big((1+\e\vp_\e)\vt_\e'\big)'',
\quad x\in \Om_{\e,j}\setminus \Om^{\frac{1}{4}\eta^{1/4}}. %D%|\z^{(j)}|<\frac{1}{4}\eta^{1/4}.
\end{equation*}
Together with (\ref{3.51}), (\ref{1.21}) it implies
\begin{equation}\label{3.52}
\|\D W\|_{L_2\left(\left\{x: |\z^{(j)}|<\frac{1}{4}\eta^{1/4}\right\}\right)} \leqslant C(K+\mu)\eta^{1/2}.
\end{equation}
It remains to obtain the estimate in the intermediate region $\Om^{\frac{3}{4}\eta^{1/4}}\setminus\Om^{\frac{1}{4}\eta^{1/4}}$. Here by the matching performed above and, in particular, by (\ref{3.35}), (\ref{3.39}), (\ref{3.46}), (\ref{3.50}), (\ref{3.47}), the estimates
\begin{align*}
&W_j^{\mathrm{mat}}(x,\e)=\Odr\Big(\e(K+\mu)\big(1+\e\vp_\e(x_1)\big) (|\xi|^2+|\vs|^{-1})\Big),
\\
&\nabla W_j^{\mathrm{mat}}(x,\e)=\Odr\big((K+\mu)(|\xi|+\eta^{-1}|\vs|^{-2})\big) =\Odr\big((K+\mu)\eta^{1/4}%D%(\eta^{1-\b}+\eta^{2\b-1})
\big)
\end{align*}
are valid. It implies
\begin{equation*}
\big\|2  \nabla \chi(|\vs^{(j)}|\eta^{3/4})\cdot\nabla  W_j^{\mathrm{mat}} +W_j^{\mathrm{mat}}\D \chi(|\vs^{(j)}|\eta^{3/4})\big\|_{L_2\left(\Om^{\frac{3}{4}\eta^{1-\b}} \setminus\Om^{\frac{1}{4}\eta^{1-\b}}\right)}\leqslant C(K+\mu)^2 \eta^{1/2},
\end{equation*}
while the other terms in $\D W$ can be estimated in the same way as in (\ref{3.49a}), (\ref{3.52}). The last estimate, (\ref{3.49a}), and (\ref{3.52}) imply the desired estimate.
\end{proof}

\begin{lemma}\label{lm3.6}
Each function $u\in\H^2(\Om)$ belongs to $C(\overline{\Om})$ and satisfies the estimate
\begin{equation*}
\sum\limits_{j\in\mathds{Z}} \|u\|_{C(\overline{\Om}_{\e,j})}^2 \leqslant C\e^{-1}\|u\|_{\H^2(\Om)}^2,
\end{equation*}
where the constant $C$ is independent of $\e$ and $u$.
\end{lemma}

\begin{proof}
By the standard smoothness improving theorems (see, for instance, \cite[Ch. I\!I, Sec. 6, Th. 3]{Mi}) each function $u\in\H^2(\Om)$ belongs to $C(\overline{\Om})$. Moreover, we let \begin{equation*}
\Om_j:=\left\{(\xi_1,x_2): |\xi_1-\pi j|<\frac{\pi}{2},\ 0<x_2<\frac{\pi}{2}\right\},\quad \widetilde{u}(\xi_1,x_2):=u(x_1\e^{-1},x_2).
\end{equation*}
Again by \cite[Ch. I\!I, Sec. 6, Th. 3]{Mi} there exists a constant $C$ independent of $\e$, $j$, and $u$ such that
\begin{equation*}
\|\widetilde{u}\|_{C(\overline{\Om}_j)}^2\leqslant C\|\widetilde{u}\|_{\H^2(Q)}\leqslant C\Big(\|\widetilde{u}\|_{\H^2(\Om_{j-1})}^2+
\|\widetilde{u}\|_{\H^2(\Om_j)}^2+
\|\widetilde{u}\|_{\H^2(\Om_{j+1})}^2\Big),
\end{equation*}
where $Q$ is a domain with infinitely differentiable boundary satisfying
\begin{equation*}
\Om_{j}\subset Q \subset \overline{\Om_{j-1}\cup\Om_j\cup\Om_{j+1}}.
\end{equation*}
Since
\begin{equation*}
\|\widetilde{u}\|_{\H^2(\Om_j)}^2 \leqslant \e^{-1} \|u\|_{\H^2(\Om_{\e,j})}^2,
\end{equation*}
we get
\begin{equation*}
\|u\|_{C^2(\overline{\Om}_{\e,j})}^2\leqslant C\e^{-1}\Big(\|\widetilde{u}\|_{\H^2(\Om_{\e,j-1})}^2+
\|\widetilde{u}\|_{\H^2(\Om_{\e,j})}^2+
\|\widetilde{u}\|_{\H^2(\Om_{\e,j+1})}^2\Big).
\end{equation*}
We sum this estimate over $j\in\mathds{Z}$ and it completes the proof.
\end{proof}

Employing Lemmas~\ref{lm3.3},~\ref{lm3.2} and proceeding as in \cite[Eqs. (3.26)-(3.32)]{BBC1}, we obtain
\begin{align}
&\|u^{(\mu)}\|_{\H^2(\Om)}\leqslant C\|f\|_{L_2(\Om)},\label{3.22}
\\
&|(f,Wv_\e)_{L_2(\Om)}|\leqslant C\e (K+\mu) |\ln\e(K+\mu)| \|f\|_{L_2(\Om)}\|v_\e\|_{\H^1(\Om)},\label{3.21}
\\
&\|u^{(\mu)}W\|_{L_2(\Om)}+\|W\nabla u^{(\mu)}\|_{L_2(\Om)} \leqslant C\e(K+\mu) |\ln\e(K+\mu)|\|f\|_{L_2(\Om)},\label{3.23}
\\
&
\begin{aligned}
\big|(u^{(\mu)}W,v_\e)_{L_2(\Om)}+&(W(\nabla+\iu A) u^{(\mu)},(\nabla+\iu A) v_\e)_{L_2(\Om)}\big|
\\
&\leqslant C\e(K+\mu) |\ln\e(K+\mu)| \|f\|_{L_2(\Om)} \|v_\e\|_{\H^1(\Om)}.
\end{aligned}
\label{3.25}
\end{align}
The inequalities (\ref{3.22}), (\ref{3.21}), (\ref{3.23}) also imply
\begin{align}
&\|W v_\e\|_{L_2(\Om)}\leqslant C\e(K+\mu) |\ln\e(K+\mu)| \|f\|_{L_2(\Om)}\|v_\e\|_{\H^1(\Om)},\label{3.26}
\\
\label{3.27}
&\big|(Vu^{(\mu)},W v_\e)_{L_2(\Om)}\big|\leqslant C\e(K+\mu) |\ln\e(K+\mu)| \|f\|_{L_2(\Om)} \|v_\e\|_{L_2(\Om)}.
\end{align}

To estimate the term $\|u^{(\mu)}\D W\|_{L_2(\Om)}$ we employ Lemmas~\ref{lm3.4},~\ref{lm3.6} and (\ref{3.22}),
\begin{equation}\label{3.28}
\begin{aligned}
\|u^{(\mu)}\D W\|_{L_2(\Om)}=&\left(\sum\limits_{j\in\mathds{Z}} \|u^{(\mu)}\D W\|_{L_2(\Om_{\e,j})}^2\right)^{1/2}
\\
&\leqslant C\e(K+\mu) \left(\sum\limits_{j\in\mathds{Z}} \|u^{(\mu)}\|_{C(\overline{\Om}_{\e,j})}^2
\right)^{1/2}
\\
&\leqslant C\e^{1/2}(K+\mu) \|u^{(\mu)}\|_{\H^2(\Om)} \leqslant C\e^{1/2}(K+\mu) \|f\|_{L_2(\Om)}.
\end{aligned}
\end{equation}
We employ   Lemmas~\ref{lm3.3},~\ref{lm3.2} with $\d=\e^2(K+\mu)^2$, the embedding of $\H^1(\Om)$ into $L_2(\G_-)$ and obtain
\begin{equation}\label{3.32}
\begin{aligned}
\|u^{(\mu)}W\|_{L_2(\G_\e)}^2=&\|u^{(\mu)}W\|_{L_2\left(\G_\e\setminus \g_\e^{\e^2(K+\mu)^2}\right)}^2 +\|u^{(\mu)}W\|_{L_2\left(\g_\e^{\e^2(K+\mu)^2}\right)}^2
\\
\leqslant& C\e^2(K+\mu)^2 |\ln\e(K+\mu)|^2 \|u^{(\mu)}\|_{L_2(\G_\e)}^2
\\
&+ C\e^2(K+\mu)^2 \|u^{(\mu)}\|_{\H^2(\Om)}^2
\\
\leqslant & C\e^2(K+\mu)^2 |\ln\e(K+\mu)|^2 \|f\|_{L_2(\Om)}^2.
\end{aligned}
\end{equation}
Hence, by \ref{asA1}, (\ref{3.50}), and the boundedness of $b$, we have
\begin{align*}
\Big|\Big(&\big(2b+(K+\mu)(1+\e\vp_\e')\vt_\e'\big) u^{(\mu)}, W v_\e\Big)_{L_2(\G_\e)}
\Big|\leqslant \|u^{(\mu)}W\|_{L_2(\G_\e)} \|v_\e\|_{L_2(\G_-)} \\
&\leqslant C\|u^{(\mu)}W\|_{L_2(\G_\e)} \|v_\e\|_{\H^1(\Om)}
\leqslant C\e(K+\mu) |\ln\e(K+\mu)| \|f\|_{L_2(\Om)} \|v_\e\|_{\H^1(\Om)}.
\end{align*}
We substitute the last estimate and (\ref{3.21}), (\ref{3.23}), (\ref{3.25}), (\ref{3.26}), (\ref{3.27}), (\ref{3.28}) into (\ref{3.3}). Then we take the real and imaginary parts of the obtained estimate and we use the obvious inequality
\begin{equation}\label{3.53}
\|\nabla v\|_{L_2(\Om)}^2\leqslant C\big(\|(\nabla +\iu A)v\|_{L_2(\Om)}^2+\|v\|_{L_2(\Om)}^2\big)
\end{equation}
which is  valid for any $v\in\H^1(\Om)$. As the result we get
\begin{align*}
&\|v_\e\|_{L_2(\Om)}^2\leqslant C\e(K+\mu)|\ln\e(K+\mu)| \|f\|_{L_2(\Om)}\|v_\e\|_{\H^1(\Om)},
\\
&\|\nabla v_\e\|_{L_2(\Om)}^2\leqslant C\e(K+\mu)|\ln\e(K+\mu)| \|f\|_{L_2(\Om)}\|v_\e\|_{\H^1(\Om)}+C\|\nabla v_\e\|_{L_2(\Om)}^2.
\end{align*}
Thus,
\begin{equation*}
\|v_\e\|_{\H^1(\Om)}\leqslant  C\e(K+\mu)|\ln\e(K+\mu)| \|f\|_{L_2(\Om)},
\end{equation*}
and it proves (\ref{1.12}).

In order to prove the other estimates we shall make use of one more auxiliary lemma.

\begin{lemma}\label{lm3.5}
The estimate
\begin{equation*}%\l%abel{3.29}
\|\nabla (u^{(\mu)}W)\|_{L_2(\Om)}\leqslant C(K+\mu)^{1/2}\|f\|_{L_2(\Om)}
\end{equation*}
holds true.
\end{lemma}

\begin{proof}
We integrate by parts taking into consideration (\ref{1.7}), and (\ref{3.1})
\begin{align*}
\|&\nabla (u^{(\mu)}W)\|_{L_2(\Om)}^2=-\left( \frac{\p}{\p x_2}
u^{(\mu)} W, u^{(\mu)}W \right)_{L_2(\G_-)}- \left(\D (u^{(\mu)}
W),u^{(\mu)}W\right)_{L_2(\Om)}
\\
=&-\left(u^{(\mu)} \frac{\p W}{\p x_2}
W, u^{(\mu)} \right)_{L_2(\g_\e)}
-\left(u^{(\mu)} \frac{\p W}{\p x_2}, u^{(\mu)}W \right)_{L_2(\G_\e)}
\\
&
-\left(W \frac{\p u^{(\mu)}}{\p x_2}, u^{(\mu)}W \right)_{L_2(\G_-)}
- \left(W\D u^{(\mu)} W, u^{(\mu)}W\right)_{L_2(\Om)}
\\
&-\left(u^{(\mu)}\D W, u^{(\mu)}W\right)_{L_2(\Om)}-2\left(W\nabla u^{(\mu)}, u^{(\mu)}\nabla W\right)_{L_2(\Om)}
\\
=&\int\limits_{\g_\e}
|u^{(\mu)}|^2\frac{\p W}{\p x_2}\di x_1+ (K+\mu)
(u^{(\mu)},u^{(\mu)}W)_{L_2(\G_\e)}
\\
&- \big( (-\iu A_2+ b+ K+\mu
)u^{(\mu)} W, u^{(\mu)} W\big)_{L_2(\G_-)}
- \left(W\D u^{(\mu)} W, u^{(\mu)}W\right)_{L_2(\Om)}
\\
&-\left(u^{(\mu)}\D W, u^{(\mu)}W\right)_{L_2(\Om)}-2\left(W\nabla u^{(\mu)}, u^{(\mu)}\nabla W\right)_{L_2(\Om)}.
\end{align*}
We take the real part of this identity
\begin{align*}
\|&\nabla (u^{(\mu)}W)\|_{L_2(\Om)}^2=-\left( \frac{\p}{\p x_2}
u^{(\mu)} W, u^{(\mu)}W \right)_{L_2(\G_-)}- \left(\D (u^{(\mu)}
W),u^{(\mu)}W\right)_{L_2(\Om)}
\\
=&\int\limits_{\g_\e}
|u^{(\mu)}|^2\frac{\p W}{\p x_2}\di x_1+ (K+\mu)
(u^{(\mu)},u^{(\mu)}W)_{L_2(\G_\e)}
\\
&- \big( (b+ K+\mu
)u^{(\mu)} W, u^{(\mu)} W\big)_{L_2(\G_-)}
- \RE\left(W\D u^{(\mu)} W, u^{(\mu)}W\right)_{L_2(\Om)}
\\
&-\RE\left(u^{(\mu)}\D W, u^{(\mu)}W\right)_{L_2(\Om)}-2\RE\left(W\nabla u^{(\mu)}, u^{(\mu)}\nabla W\right)_{L_2(\Om)}.
\end{align*}
We calculate the last term in the right hand side of this equation by one more integration by parts. Using (\ref{1.7}) and (\ref{3.1}), we get
\begin{align*}
-&2\RE(W\nabla u^{(\mu)},u^{(\mu)}\nabla W)_{L_2(\Om)}=-\frac{1}{2} \int\limits_{\Om} \nabla W^2\cdot \nabla |u^{(\mu)}|^2\di x
\\
&=\frac{1}{2} \int\limits_{\G_-} W^2 \frac{\p}{\p x_2} |u^{(\mu)}|^2\di x_1+\frac{1}{2}\int\limits_{\Om} W^2\D |u^{(\mu)}|^2\di x
\\
&=\big((b+K+\mu) u^{(\mu)}W,u^{(\mu)}W\big)_{L_2(\G_-)} + \RE (u^{(\mu)}W, W\D u^{(\mu)})_{L_2(\Om))} + \|W\nabla u^{(\mu)}\|_{L_2(\Om)}^2.
\end{align*}
We substitute the obtained identity into the previous one and we find
\begin{equation}\label{3.31}
\begin{aligned}
\|\nabla (u^{(\mu)}W)\|_{L_2(\Om)}^2=&\int\limits_{\g_\e} |u^{(\mu)}|^2 \frac{\p W}{\p x_2}\di x_1+ (K+\mu) (u^{(\mu)},u^{(\mu)}W)_{L_2(\G_\e)}
\\
&-\RE (u^{(\mu)}\D W, u^{(\mu)}W)_{L_2(\Om)}+\|W\nabla u^{(\mu)}\|_{L_2(\Om)}^2.
\end{aligned}
\end{equation}
It follows from (\ref{3.23}) and  (\ref{3.28}) that
\begin{equation*}
 \big|(u^{(\mu)}\D W, u^{(\mu)}W)_{L_2(\Om)}\big| \leqslant C\e^{3/2} (K+\mu)^2 |\ln\e(K+\mu)| \|f\|_{L_2(\Om)}^2.
\end{equation*}
Following the arguments of the proof of Lemma~3.5 in \cite{BBC1}, one can check easily that
\begin{equation*}
\left|\int\limits_{\g_\e} |u^{(\mu)}|^2 \frac{\p W}{\p x_2}\di x_1\right| \leqslant C(K+\mu)\|f\|_{L_2(\Om)}^2.
\end{equation*}
Due to (\ref{3.32}) we have
\begin{equation*}
\big|(u^{(\mu)},u^{(\mu)}W)_{L_2(\G_\e)}\big|\leqslant C\e(K+\mu) |\ln\e(K+\mu)| \|f\|_{L_2(\Om)}.
\end{equation*}
We substitute the three last estimates and (\ref{3.23}) into (\ref{3.31}) and it completes the proof.
\end{proof}

It is clear that the resolvent $(\mathcal{H}_{\mathrm{R}}^{(\mu)}+\iu)^{-1}$  is analytic in $\mu$ and this is why
\begin{equation}\label{3.33}
\|(\mathcal{H}_{\mathrm{R}}^{(\mu)}+\iu)^{-1} -(\mathcal{H}_{\mathrm{R}}^{(0)}+\iu)^{-1}\|_{L_2(\Om)\to \H^1(\Om)}\leqslant C\mu.
\end{equation}
This inequality,  Lemma~\ref{lm3.5}, the estimates (\ref{1.12}), and (\ref{3.23}) yield the estimate (\ref{1.13}).   The estimates (\ref{1.8}) and
 (\ref{1.9}) follow from (\ref{1.12}), (\ref{3.23}) and  (\ref{3.33}).

\section{Resolvent convergence: homogenized Dirichlet condition}

In this section we study the uniform resolvent convergence for the case of the homogenized Dirichlet condition and prove Theorem~\ref{th1.4}. The case of periodic alternation $\vt_\e(t)\equiv t$, $a_j^\pm\equiv\eta$ for the pure Laplacian $\mathcal{H}_\e=-\D$ was proved in \cite[Sec.2]{BC}. Here we employ similar ideas but with changes required by the non-periodicity of the alternation and for the more general operator.

Given $f\in L_2(\Om)$, let us denote
\begin{equation*}
u_\e:=(\mathcal{H}_\e-\iu)^{-1}f, \quad u_0:=(\mathcal{H}_{\mathrm{D}}-\iu)^{-1}f, \quad v_\e:=u_\e-u_0.
\end{equation*}
The last function belongs to $\Ho^1(\Om,\G_+\cup\g_\e)$ and it is the generalized solution of the boundary value problem
\begin{gather*}
\big((\iu\nabla +A)^2+V-\iu\big) v_\e=0\quad\text{in}\quad \Om,\\ v_\e=0\quad\text{on}\quad \G_+\cup\g_\e,\qquad \left(-\frac{\p}{\p x_2}-\iu A_2+b\right) v_\e=\frac{\p u_0}{\p x_2}\quad\text{on}\quad \G_\e.
\end{gather*}
We multiply this equation by $\overline{v}_\e$, %D%we
integrate by parts over $\Om$ taking into consideration the boundary conditions, and %D%we
get
\begin{equation}
\begin{aligned}
0=& \int\limits_{\G_-} \overline{v}_\e \left(\frac{\p}{\p x_2}+\iu A_2\right)v_\e\di x_1+\big\|(\iu\nabla+A)v_\e\big\|_{L_2(\Om)}^2+ (V v_\e, v_\e)_{L_2(\Om)}-\iu \|v_\e\|_{L_2(\Om)}^2
\\
=& \big\|(\iu\nabla+A)v_\e\big\|_{L_2(\Om)}^2+ (V v_\e, v_\e)_{L_2(\Om)}-\iu \|v_\e\|_{L_2(\Om)}^2+(b v_\e, v_\e)_{L_2(\G_\e)}^2 - \left(\frac{\p u_0}{\p x_2}, v_\e\right)_{L_2(\G_\e)}.
\end{aligned}\label{5.1}
\end{equation}
Since $u_0\in\H^2(\Om)$, the term $\left(\frac{\p u_0}{\p x_2}, v_\e\right)_{L_2(\G_\e)}$ is well-defined by the embedding of $\H^2(\Om)$ into $\H^1(\G_-)$ and
\begin{equation*}
\left|\left(\frac{\p u_0}{\p x_2}, v_\e\right)_{L_2(\G_\e)}\right|\leqslant C\Big\|\frac{\p u_0}{\p x_2}\Big\|_{L_2(\G_\e)} \|v_\e\|_{L_2(\G_\e)} \leqslant C\|u_0\|_{\H^2(\Om)}\|v_\e\|_{L_2(\G_\e)}.
\end{equation*}
Here and till the end of the proof we denote by $C$ inessential constants independent of $\e$ and $f$. The last estimate, (\ref{5.1}), (\ref{3.53}), and the identity $v_\e=u_\e$ on $\G_\e$ yield
\begin{equation}\label{5.2}
\|v_\e\|_{\H^1(\Om)}^2\leqslant C\left( \|u_\e\|_{L_2(\G_\e)}^2+ \|u_0\|_{\H^2(\Om)} \|u_\e\|_{L_2(\G_\e)}\right).
\end{equation}

We employ again the rescaled variables $\xi$ defined in (\ref{3.34}) and introduce one more auxiliary function
\begin{equation*}
X_\eta=X_\eta(\xi):=\RE\ln\big(\sin\rho+\sqrt{\sin^2\rho-\sin^2\eta}
\big)-\xi_2,\quad \rho=\xi_1+\iu\xi_2,
\end{equation*}
where the   branches of the logarithm and the square root are fixed by the requirements $\ln 1=0$, $\sqrt{1}=1$. This function was introduced in \cite{G-SPMJ98}. It was proven that it is harmonic in the half-space $\xi_2>0$, even and $\pi$-periodic in $\xi_1$, decays exponentially as $\xi_2\to+\infty$, and satisfies the boundary condition
%\begin{equation}\l%abel{5.3}
\begin{align*}
&X_\eta=\ln\sin\eta &&\text{on}\quad \g(\eta):=\bigcup\limits_{j\in\mathds{Z}} \{\xi: |\xi_1-\pi j|<\eta,\ \xi_2=0\},
\\
&\frac{\p X_\eta}{\p \xi_2}=-1 &&\text{on}\quad \G(\eta):=O\xi_1\setminus\overline{\g(\eta)}.
\end{align*}
%\end{equation}
In \cite[Sec. 3]{BC} it was shown that the function $X_\eta$ is continuous in $\{\xi: \xi_2\geqslant 0\}$ and obeys the estimate
\begin{equation}\label{5.7}
|X_\eta(\xi)|\leqslant |\ln\sin\eta|
\end{equation}
uniformly for $\xi_2\geqslant 0$.

Denote
\begin{equation*}
\eta_*(\e):=\min\{\inf\limits_j a_j^+(\e), \inf\limits_j a_j^-(\e)\}.
\end{equation*}
By the assumption (A3) we get
\begin{equation}\label{5.4}
\eta(\e)\leqslant \eta_*(\e)\leqslant c_3\eta(\e)\leqslant \frac{\pi}{2}.
\end{equation}
In the same way as in Lemma~3.2 in \cite{BC}, one can check easily that $u_\e X_{\eta_*(\e)}(\cdot\,\e^{-1})\in\Ho^1(\Om,\G_+\cup\g_\e)$. We choose this product as the test function in the integral identity for $u_\e$,
\begin{equation*}
\mathfrak{h}_\e[u_\e,u_\e X_{\eta_*}]=(f,u_\e X_{\eta_*})_{L_2(\Om)},
\end{equation*}
or, by (\ref{1.4}),
\begin{equation}\label{5.5}
\begin{aligned}
\big(&(\iu\nabla+A)u_\e,X_{\eta_*}(\iu\nabla+A) u_\e\big)_{L_2(\Om)} + \big((\iu\nabla+A)u_\e,\iu u_\e \nabla X_{\eta_*}\big)_{L_2(\Om)}
\\
&+(V u_\e,X_{\eta_*}u_\e)_{L_2(\Om)}+
(b u_\e,u_\e)_{L_2(\Gamma_{\e})} -\iu (u_\e, X_{\eta_*} u_\e)_{L_2(\Om)} = (f,u_\e X_{\eta_*})_{L_2(\Om)},
\end{aligned}
\end{equation}
where $X_{\eta_*}=X_{\eta_*}(\cdot\,\e^{-1})$.

Employing the described properties of $X_\eta$, we integrate by parts as follows,
\begin{align*}
&\RE \big((\iu\nabla+A) u_\e, \iu u_\e \nabla X_{\eta_*}\big)_{L_2(\Om)} =\RE (\nabla  u_\e, u_\e  \nabla X_{\eta_*})_{L_2(\Om)}
\\
&= \frac{1}{2}\int\limits_{\Om} \nabla X_{\eta_*}\cdot \nabla |u_\e|^2\di x =
- \frac{1}{2}\int\limits_{\G_-} |u_\e|^2\frac{\p X_{\eta_*}}{\p x_2} \di x_1
- \int\limits_{\Om} |u_\e|^2\D X_{\eta_*}\di x
\\
&=\frac{1}{2\e} \int\limits_{\G_\e} \vt_\e'|u_\e|^2\di x_1- \int\limits_{\Om} |u_\e|^2\D X_{\eta_*}\di x.
\end{align*}
This identity, (\ref{1.20}), and the real part of (\ref{5.5}) imply
\begin{align*}
\|u_\e\|_{L_2(\G_\e)}^2&+2\e (b u_\e, X_{\eta_*}u_\e)_{L_2(\G_\e)}=
\e\int\limits_{\Om} |u_\e|^2\D X_{\eta_*}\di x+2\e (f, X_{\eta_*}u_\e)_{L_2(\Om)}
\\
&-2\e \big((\iu\nabla+A)u_\e,X_{\eta_*}(\iu\nabla+A) u_\e\big)_{L_2(\Om)}-2\e (V u_\e,X_{\eta_*}u_\e)_{L_2(\Om)}.
\end{align*}
By the boundedness of $b$, $V$, and $A$ it follows  that
\begin{equation}\label{5.6}
\|u_\e\|_{L_2(\G_\e)}^2\leqslant C\e \left( \left|\int\limits_{\Om} |u_\e|^2\D X_{\eta_*}\di x\right|+ |\ln\sin\eta_*(\e)| \Big(\|u_\e\|_{\H^1(\Om)}^2+\|f\|_{L_2(\Om)}^2\Big)\right).
\end{equation}
Let us estimate the first term in the right hand side of the last inequality.

In view of the harmonicity of $X_{\eta_*}$ and the definition
(\ref{3.34}) of $\xi$ we get
\begin{align*}
\D X_{\eta_*}=2\frac{\vt_\e'}{\e} x_2\frac{\p}{\p x_1}\frac{\p X_{\eta_*}}{\p\xi_2}+\frac{\p}{\p x_1} X_{\eta_*} -\left(\frac{2\vt_\e''}{\vt_\e'}+ \left(\frac{\vt_\e''}{\vt_\e'}\right)^2\right) \xi_2^2\frac{\p^2 X_{\eta_*}}{\p\xi_2^2}+ \frac{\vt_\e'''-\vt_\e''}{\vt_\e'} \xi_2\frac{\p X_{\eta_*}}{\p\xi_2}.
\end{align*}
Hence, we can integrate by parts as follows, %D%
\begin{align*}
\int\limits_{\Om} |u_\e|^2\D X_{\eta_*}\di x=&-2\int\limits_{\Om} X_{\eta_*}\frac{\xi_2}{\vt_\e'} \frac{\p}{\p x_1} \vt_\e'|u_\e|^2\di x- \int\limits_{\Om} X_{\eta_*} \frac{\p}{\p x_1}  |u_\e|^2\di x
\\
&+ \int\limits_{\Om} |u_\e|^2 \left( \frac{\vt_\e'''-\vt_\e''}{\vt_\e'} \xi_2\frac{\p }{\p\xi_2}-\left(\frac{2\vt_\e''}{\vt_\e'}+ \left(\frac{\vt_\e''}{\vt_\e'}\right)^2\right) \xi_2^2\frac{\p^2 }{\p\xi_2^2}  \right)X_{\eta_*}\di x,
\end{align*}
and thus
\begin{equation}\label{5.8}
\left|\int\limits_{\Om} |u_\e|^2\D X_{\eta_*}\di x\right| \leqslant C\left(\sup\limits_{\xi_2\geqslant 0} \Big|\xi_2^2 \frac{\p^2 X_{\eta_*}}{\p\xi_2^2}\Big| + \sup\limits_{\xi_2\geqslant 0} \Big|\xi_2 \frac{\p X_{\eta_*}}{\p\xi_2}\Big|+\sup\limits_{\xi_2\geqslant 0} |X_{\eta_*}|\right)\|u_\e\|_{\H^1(\Om)}^2.
\end{equation}

\begin{lemma}\label{lm5.1}
The estimate
\begin{equation*}%\l%abel{5.9}
\sup\limits_{\xi_2\geqslant 0} \Big|\xi_2^2 \frac{\p^2 X_{\eta_*}}{\p\xi_2^2}\Big|+ \sup\limits_{\xi_2\geqslant 0} \Big|\xi_2 \frac{\p X_{\eta_*}}{\p\xi_2}\Big|\leqslant C\cos\eta_*
\end{equation*}
holds true.
\end{lemma}
\begin{proof}
Since the function $X_{\eta_*}$ is even and $\pi$-periodic in $\xi_1$, we have
\begin{equation}\label{5.10}
\sup\limits_{\xi_2\geqslant 0} \Big|\xi_2^2 \frac{\p^2 X_{\eta_*}}{\p\xi_2^2}\Big|+ \sup\limits_{\xi_2\geqslant 0} \Big|\xi_2 \frac{\p X_{\eta_*}}{\p\xi_2}\Big|=\sup\limits_{\genfrac{}{}{0 pt}{}{0\leqslant \xi_1\leqslant \pi/2}{\xi_2\geqslant 0}} \Big|\xi_2^2 \frac{\p^2 X_{\eta_*}}{\p\xi_2^2}\Big|+ \sup\limits_{\genfrac{}{}{0 pt}{}{0\leqslant \xi_1\leqslant \pi/2}{\xi_2\geqslant 0}}  \Big|\xi_2 \frac{\p X_{\eta_*}}{\p\xi_2}\Big|.
\end{equation}
It follows from the definition of $X_{\eta_*}$ that
\begin{align*}
\frac{\p X_{\eta_*}}{\p\xi_2}=&-\IM \frac{\cos\rho}{\sqrt{\sin^2\rho-\sin^2\eta}}-1=-\IM \frac{\cos\rho+\iu\sqrt{\sin^2\rho-\sin^2\eta}}{\sqrt{\sin^2\rho-\sin^2\eta}} \\
=&-\IM \frac{\cos\eta_*}{\sqrt{\sin^2\rho-\sin^2\eta}\left(
\frac{\cos\rho}{\cos\eta_*}+\sqrt{\frac{\cos^2\rho}{\cos^2\eta_*}-1}
\right)}.
\end{align*}
As $0\leqslant \xi_1\leqslant \pi/2$, $\xi_2\geqslant 0$, the fraction $\frac{\cos\rho}{\cos\eta_*}$ ranges in the  closure of the first quarter of the complex plane. In this quarter the function $z\mapsto z+\sqrt{z^2-1}$ has no zeroes. Moreover, as $z$ tends to infinity within this quarter, the function $z+\sqrt{z^2-1}$ tends to infinity, too. Hence, there exists a positive constant $C$ such that
\begin{equation*}
|z+\sqrt{z^2-1}|\geqslant C>0\quad\text{uniformly in}\quad \RE z\geqslant 0,\quad \IM z\geqslant 0.
\end{equation*}
This estimate implies
\begin{equation*}
\left|\frac{\cos\rho}{\cos\eta_*}+ \sqrt{\frac{\cos^2\rho}{\cos^2\eta_*}-1}\right| \geqslant C,\quad 0\leqslant \xi_1\leqslant \pi/2, \quad \xi_2\geqslant 0,
\end{equation*}
and therefore
\begin{equation}\label{5.11}
\left|\xi_2\frac{\p X_{\eta_*}}{\p \xi_2}\right|\leqslant \frac{C\xi_2\cos\eta_*}{|\sin^2\rho-\sin^2\eta_*|^{1/2}},\quad
0\leqslant \xi_1\leqslant \pi/2, \quad \xi_2\geqslant 0,
\end{equation}
where the constant $C$ is independent of $\xi$ and $\eta_*$. By direct calculations we check that
\begin{align}
&
\begin{aligned}
|\sin^2\rho-\sin^2\eta_*&|^{1/2}=|\sin(\rho-\eta_*)\sin(\rho+\eta_*)|^{1/2} \\
= &\big(\sinh^2\xi_2+\sin^2(\xi_1-\eta_*)
\big)^{1/4}\big(\sinh^2\xi_2+\sin^2(\xi_1+\eta_*)
\big)^{1/4},
\end{aligned}\label{5.15}
\\
&|\sin^2\rho-\sin^2\eta_*|^{1/2}
\geqslant  \sinh\xi_2.\label{5.13}
\end{align}
Hence, by (\ref{5.11})
\begin{equation}\label{5.12}
\left|\xi_2\frac{\p X_{\eta_*}}{\p \xi_2}\right|\leqslant \frac{C\xi_2\cos\eta_*}{\sinh \xi_2}\leqslant C \cos\eta_*,\quad
0\leqslant \xi_1\leqslant \pi/2, \quad \xi_2\geqslant 0.
\end{equation}

In the same way we estimate the first term in the right hand side of (\ref{5.10}). We have
\begin{equation*}
\frac{\p^2 X_{\eta_*}}{\p\xi_2^2}=\cos^2\eta_* \RE \frac{\sin\rho }{(\sin^2\rho-\sin^2\eta_*)^{3/2}}.
\end{equation*}
Thus,
\begin{align}
\xi_2^2\frac{\p^2 X_{\eta_*}}{\p\xi_2^2}=&\xi_2^2  \cos^2\eta_* \RE \frac{1}{(\sin^2\rho-\sin^2\eta_*)^{1/2} (\sin\rho-\sin\eta_*)} \nonumber
\\
&+\xi_2^2 \cos^2\eta_*\sin\eta_* \RE\frac{1}{(\sin^2\rho-\sin^2\eta_*)^{3/2}}, \nonumber
\\
\left|\xi_2^2\frac{\p^2 X_{\eta_*}}{\p\xi_2^2}\right| \leqslant& \frac{\xi_2^2\cos^2\eta_*} {|\sin^2\rho-\sin^2\eta_*|^{1/2}|\sin\rho-\sin\eta_*|}+  \frac{\xi_2^2 \cos^2\eta_*\sin\eta_* }{|\sin^2\rho-\sin^2\eta_*|^{3/2}}.\label{5.14}
\end{align}
It is easy to see that
\begin{equation*}
|\sin\rho+\sin\eta_*|^2=\sinh^2\xi_2+\sin^2\xi_1+\sin^2\eta_* +2\sin\xi_1\cosh\xi_2\sin\eta_* \geqslant \sinh^2\xi_2
\end{equation*}
for $0\leqslant\xi_1\leqslant \pi/2$, $\xi_2\geqslant 0$. It also follows from (\ref{5.15}) that
\begin{align*}
|\sin^2\rho-\sin^2\eta_*|^{3/2}\geqslant & \big(\sinh^2\xi_2+\sin^2(\xi_1+\eta)\big)^{1/2}\sinh^2\xi_2
\\
\geqslant &\sinh^2\xi_2\min(\sin\eta_*,\cos\eta_*)\geqslant \sinh^2\xi_2\sin\eta_*\cos\eta_*.
\end{align*}
We substitute two last estimates and (\ref{5.13}) into (\ref{5.14}),
\begin{equation*}
\left|\xi_2^2\frac{\p^2 X_{\eta_*}}{\p\xi_2^2}\right| \leqslant \frac{\xi_2^2\cos^2\eta_*} {\sinh^2\xi_2}+  \frac{\xi_2^2 \cos\eta_*}{\sinh^2\xi_2}\leqslant C\cos\eta_*,\quad \xi_2\geqslant 0,\quad 0\leqslant \xi_1\leqslant \frac{\pi}{2}.
\end{equation*}
The obtained estimate, (\ref{5.12}), and (\ref{5.10}) complete the proof.
\end{proof}

The proven lemma and the estimates (\ref{5.7}), (\ref{5.8}) yield
\begin{equation*}
\left|\int\limits_{\Om} |u_\e|^2\D X_{\eta_*}\di x\right| \leqslant C\big(|\ln\sin\eta_*(\e)|+\cos\eta_*(\e)\big)\|u_\e\|_{\H^1(\Om)}^2.
\end{equation*}
Together with (\ref{5.6}) it implies
\begin{equation*}
\|u_\e\|_{L_2(\G_\e)}^2\leqslant C\e \big(|\ln\sin\eta_*(\e)|+\cos\eta_*(\e)
\big) \big( \|u_\e\|_{\H^1(\Om)}^2+\|f\|_{L_2(\Om)}^2\big).
\end{equation*}
We observe that the functions $u_0$ and $u_\e$ satisfy the estimates
\begin{equation*}
\|u_0\|_{\H^2(\Om)}\leqslant C\|f\|_{L_2(\Om)},\quad \|u_\e\|_{\H^1(\Om)}\leqslant C\|f\|_{L_2(\Om)}.
\end{equation*}
Due to (\ref{5.4}) we have
\begin{equation*}
|\ln\sin\eta_*(\e)|\leqslant C\|f\|_{L_2(\Om)},\quad \cos\eta_*(\e)\leqslant \cos\eta(\e).
\end{equation*}
It follows from last five estimates and (\ref{5.2}) that
\begin{equation*}
\|v_\e\|_{\H^1(\Om)}^2\leqslant C\e^{1/2} \big(|\ln\sin\eta(\e)|+\cos\eta(\e)\big)^{1/2}\|f\|_{L_2(\Om)}^2
\end{equation*}
and it proves Theorem~\ref{th1.4}.

\section{Analysis of the operator $\Hpe(\tau)$}

In this section we prove Theorem~\ref{th1.3}.

We state now two auxiliary lemmas proved in \cite{BC} and \cite{BBC1} that we give below for the reader's convenience.

\begin{lemma}\label{lm4.1}
Let $|\tau|<1-\k$, where $\k\in(0,1)$, and for a given function $f\in L_2(\Om_\e)$ we let
\begin{equation*}
U_\e=\left(\Hpe(\tau)-\frac{\tau^2}{\e^2}\right)^{-1}f,\quad f\in
L_2(\Om_\e).
\end{equation*}
Assume $f\in \mathfrak{L}_\e^\bot$. Then
\begin{equation*}%\l%abel{4.2}
\|U_\e\|_{L_2(\Om_\e)}\leqslant
\frac{\e}{\k^{1/2}}\|f\|_{L_2(\Om_\e)},\quad \|\nabla
U_\e\|_{L_2(\Om_\e)}\leqslant \frac{\e}{2\k} \|f\|_{L_2(\Om_\e)}.
\end{equation*}
For any $u\in\Hoper^1(\Om_\e,\Gp_+)$ and $|\tau|\leqslant 1-\k$, we have the following inequalities
\begin{equation}\label{4.3}
\begin{aligned}
&\Big\| \left(\iu \frac{\p}{\p x_1}-
\frac{\tau}{\e}\right)u\Big\|_{L_2(\Om_\e)}^2-\frac{\tau^2}{\e^2}
\|u\|_{L_2(\Om_\e)}^2\geqslant \k \Big\|\frac{\p u}{\p
x_1}\Big\|_{L_2(\Om_\e)}^2,
\\
&\Big\|\frac{\p u}{\p x_2}\Big\|_{L_2(\Om_\e)}\geqslant
\frac{1}{2}\|u\|_{L_2(\Om)}.
\end{aligned}
\end{equation}

\end{lemma}

\begin{lemma}\label{lm4.2}
If $F\in L_2(0,\pi)$, then $|(\mathcal{Q}_\mu^{-1}F)(0)|\leqslant 5\|F\|_{L_2(0,\pi)}$.
\end{lemma}
\bigskip
Let $f\in L_2(\Om_\e)$ and $f=F_\e+f_\e^\bot$, where $F_\e\in
\mathfrak{L}_\e$, $f_\e^\bot\in\mathfrak{L}_\e^\bot$,
\begin{gather}
F_\e(x_2)=\frac{1}{\pi\e^{1/2}}\int\limits_{-\frac{\e\pi}{2}}^{\frac{\e\pi}{2}}
f_\e(x)\di x_1,\quad f_\e^\bot(x):=f(x)-F_\e(x_2),
\nonumber
\\
\e\pi\|F_\e\|_{L_2(0,\pi)}^2+\|f_\e^\bot\|_{L_2(\Om_\e)}^2=
\|f\|_{L_2(\Om_\e)}^2,\quad U_\e:=\left(\Hpe(\tau)-\frac{\tau^2}{\e^2}\right)^{-1}F_\e. \label{4.4}
\end{gather}
It follows from Lemma~\ref{lm4.1} that
\begin{equation}\label{4.5}
\begin{aligned}
\bigg\|\left(\Hpe(\tau)-\frac{\tau^2}{\e^2}\right)^{-1}f-U_\e\bigg\|_{L_2(\Om_\e)}
=\bigg\|\left(\Hpe(\tau)-\frac{\tau^2}{\e^2}\right)^{-1}f_\e^\bot\bigg\|_{L_2(\Om_\e)}
\leqslant \e \k^{-1/2} \|f\|_{L_2(\Om_\e)}.
\end{aligned}
\end{equation}
Denote
\begin{equation}\label{4.12}
U^{(\mu)}_\e:=\mathcal{Q}_\mu^{-1}F_\e,
\quad V_\e(x):=U_\e(x)-U^{(\mu)}_\e(x)-U^{(\mu)}_\e(0)W(x,\e,\mu)\chi(x_2).
\end{equation}
We remind that the function $\chi$ was introduced after the equation (\ref{3.46}). The boundary conditions (\ref{3.1})  and the definition of $U_\e$ imply that $V_\e\in\Hoper^1(\Om_\e,\Gp_+\cup\gp_\e)$.

It follows from the definition of $U_\e$ and $U_\e^{(\mu)}$ that they satisfy the integral equations
\begin{align}
&\hpe(\tau)[U_\e,V_\e]-\frac{\tau^2}{\e^2}(U_\e,V_\e)_{L_2(\Om_\e)} =(F_\e,V_\e)_{L_2(\Om_\e)},\label{4.6}
\\
&\mathfrak{q}_\mu[U_\e,V_\e]-\frac{\tau^2}{\e^2} \big(
U_\e, V_\e(x_1,\cdot)\big)_{L_2(0,\pi)}=\big(F_\e, V_\e(x_1,\cdot)\big)_{L_2(0,\pi)}.\nonumber
\end{align}
We integrate the last equation over $x_1\in(-\e\pi/2,\e\pi/2)$ and use an obvious relation
\begin{align*}
\left( \left(\iu\frac{\p}{\p
x_1}-\frac{\tau}{\e}\right)U^{(\mu)}_\e, \left(\iu\frac{\p}{\p
x_1}-\frac{\tau}{\e}\right) V_\e\right)_{L_2(\Om_\e)}=&
-\frac{\tau}{\e}\left(U^{(\mu)}_\e, \left(\iu\frac{\p}{\p
x_1}-\frac{\tau}{\e}\right) V_\e\right)_{L_2(\Om_\e)}
\\
=&\frac{\tau^2}{\e^2} (U^{(\mu)}_\e, V_\e)_{L_2(\Om_\e)}.
\end{align*}
It leads us to the identity
\begin{align*}
&\left( \left(\iu\frac{\p}{\p
x_1}-\frac{\tau}{\e}\right)U^{(\mu)}_\e, \left(\iu\frac{\p}{\p
x_1}-\frac{\tau}{\e}\right) V_\e\right)_{L_2(\Om_\e)}+ \left(\frac{\p U_\e^{(\mu)}}{\p x_2},\frac{\p V_\e}{\p x_2}\right)_{L_2(\Om_\e)}
\\
&\hphantom{\iu\frac{\p}{\p
x_1}}- \frac{\tau^2}{\e^2} (U_\e^{(\mu)}, V_\e)_{L_2(\Om_\e)}+ (b+K+\mu) (U_\e^{(\mu)}, V_\e)_{L_2(\Gp_\e)}=(F_\e,V_\e)_{L_2(\Om_\e)}.
\end{align*}
We take the difference between the last identity and (\ref{4.6}), %D%; we obtain
\begin{equation*}
\hpe(\tau)[U_\e-U_\e^{(\mu)},V_\e]-\frac{\tau^2}{\e^2}(U_\e-U_\e^{(\mu)}, V_\e)_{L_2(\Om_\e)}= (K+\mu) (U_\e^{(\mu)},V_\e)_{L_2(\Gp_\e)},
\end{equation*}
and since
\begin{equation*}%\l%abel{4.7}
U_\e-U_\e^{(\mu)}=V_\e+U_\e^{(\mu)}(0)W\chi,
\end{equation*}
we get
\begin{align*}
\hpe(\tau)[V_\e,V_\e]-\frac{\tau^2}{\e^2}\|V_\e\|_{L_2(\Om_\e)}^2 =&(K+\mu) (U_\e^{(\mu)},V_\e)_{L_2(\Gp_\e)}
\\
&-U_\e^{(\mu)}(0)\hpe(\tau)[W\chi,V_\e]-\frac{\tau^2}{\e^2} (W\chi,V_\e)_{L_2(\Om_\e)}.
\end{align*}
Integrating by parts and employing (\ref{3.1}), we obtain
\begin{align*}
\hpe&(\tau)[W\chi,V_\e]-\frac{\tau^2}{\e^2} (W\chi,V_\e)_{L_2(\Om_\e)}
\\
=&(\nabla W\chi,\nabla V_\e)_{L_2(\Om_\e)} -\frac{\iu\tau}{\e} \left(\frac{\p W \chi}{\p x_1}, V_\e\right)_{L_2(\Om_\e)}
  + \frac{\iu \tau}{\e} \left(W\chi,\frac{\p V_\e}{\p x_1}\right)_{L_2(\Om_\e)} + b(W,V_\e)_{L_2(\Gp_\e)}
\\
= &-\int\limits_{\Gp_-} \overline{V}_\e \frac{\p W\chi}{\p x_2}\di x_1- (\D W\chi,V_\e)_{L_2(\Om_\e)}+ \frac{2\iu\tau}{\e} \left(W\chi,\frac{\p V_\e}{\p x_1}\right)_{L_2(\Om_\e)} + b(W,V_\e)_{L_2(\Gp_\e)}
\\
=&-(\D W\chi,V_\e)_{L_2(\Om_\e)}+ \frac{2\iu\tau}{\e} \left(W\chi,\frac{\p V_\e}{\p x_1}\right)_{L_2(\Om_\e)}+ (bW+K+\mu,V_\e)_{L_2(\Gp_\e)}.
\end{align*}
The last two equations imply
\begin{equation}\label{4.8}
\begin{aligned}
\hpe(\tau)&[V_\e,V_\e]-\frac{\tau^2}{\e^2}\|V_\e\|_{L_2(\Om_\e)}^2= -U_\e^{(\mu)}(0) (bW,V_\e)_{L_2(\Gp_\e)}
\\
&+ U_\e^{(\mu)}(0) (\D W\chi,V_\e)_{L_2(\Om_\e)}-\frac{2\iu\tau}{\e} U_\e^{(\mu)}(0) \left(W\chi, \frac{\p V_\e}{\p x_1}\right)_{L_2(\Om_\e)}.
\end{aligned}
\end{equation}
Lemma~\ref{lm4.2} and the inequalities (\ref{4.4}) give the estimate for $U_\e^{(\mu)}(0)$,
\begin{equation}\label{4.9}
|U_\e^{(\mu)}(0)|\leqslant 5(\pi\e)^{-1/2} \|f\|_{L_2(\Om_\e)}.
\end{equation}
In the same way as in Theorem~2.3 in \cite{BBC1}, we get the estimates for $W\chi$ and $\D W\chi$,
\begin{equation}\label{4.10}
\begin{aligned}
&\|W\chi\|_{L_2(\Om_\e)}\leqslant \|W\|_{L_2(\Om_\e)}\leqslant C\e^2(K+\mu),
\\
&\|\D W\chi\|_{L_2(\Om_\e)} \leqslant C(K+\mu) (\eta^{1/4}+\E^{-\e^{-1}}).
\end{aligned}
\end{equation}
Using the last three inequalities, we estimate the second and the third term in the right hand side of (\ref{4.8}) as
\begin{equation}
\begin{aligned}
&\left|U_\e^{(\mu)}(0)(\D W\chi,V_\e)_{L_2(\Om_\e)} -\frac{2\iu\tau}{\e} U_\e^{(\mu)}(0) \left(W\chi,\frac{\p V_\e}{\p x_1}\right)_{L_2(\Om_\e)}\right|
\\
&\hphantom{U_\e^{(\mu)}}\leqslant 5(\pi\e)^{-1/2} \|f\|_{L_2(\Om)} \Bigg(\|\D W\chi\|_{L_2(\Om_\e)} \|V_\e\|_{L_2(\Om_\e)}
\\
&\hphantom{U_\e^{(\mu)}\leqslant 5(\pi\e)^{-1/2} \|f\|_{L_2(\Om)} \Bigg(}+ 2\e^{-1} \|W\chi\|_{L_2(\Om_\e)}\left\|\frac{\p V_\e}{\p x_1}\right\|_{L_2(\Om_\e)}\Bigg)
\\
&\hphantom{U_\e^{(\mu)}}\leqslant \frac{\k}{2} \left\|\frac{\p V_\e}{\p x_1}\right\|_{L_2(\Om_\e)}^2 + \frac{1}{16} \|V_\e\|_{L_2(\Om_\e)}^2 + C\k^{-1}\e(K+\mu)^2 \|f\|_{L_2(\Om_\e)}^2.
\end{aligned}\label{4.11}
\end{equation}

It remains to estimate the first term in the right hand side of (\ref{4.8}). We employ the obvious inequality
\begin{equation*}
\|V_\e\|_{L_2(\Gp_\e)}\leqslant C\left\|\frac{\p V_\e}{
\p x_2}\right\|_{L_2(\Om_\e)}
\end{equation*}
(\ref{3.19}), and (\ref{4.9}) to obtain
\begin{align*}
\big|U_\e^{(\mu)}(0)& (bW,V_\e)_{L_2(\Gp_\e)}
\big|\leqslant C|U_\e^{(\mu)}(0)|\|bW\|_{L_2(\Gp_\e)} \|V_\e\|_{L_2(\Gp_\e)}
\\
&\leqslant C\e^{-1/2}\|f\|_{L_2(\Om_\e)} \left\|\frac{\p V_\e}{\p x_2}\right\|_{L_2(\Om_\e)} \left(\|W\|_{L_2(\Gp_\e\setminus\g_\e^{2\eta^{3/4}})} +\|W\|_{L_2(\Gp_\e\cap\g_\e^{2\eta^{3/4}})}
\right)
\\
&\leqslant \frac{1}{2}\left\|\frac{\p V_\e}{\p x_2}\right\|_{L_2(\Om_\e)}^2 + C\e(K+\mu)^2|\ln\e(K+\mu)|^2 \|f\|_{L_2(\Om_\e)}^2.
\end{align*}
We substitute this estimate and (\ref{4.11}) into (\ref{4.8}), and by (\ref{4.3}) it follows %D%we obtain
\begin{align*}
\k \left\|\frac{\p V_\e}{\p x_1}\right\|_{L_2(\Om_\e)}^2 + & \frac{1}{2} \left\|\frac{\p V_\e}{\p x_2}\right\|_{L_2(\Om_\e)}^2 + \frac{1}{8}\|V_\e\|_{L_2(\Om_\e)}^2 \leqslant \frac{\k}{2} \left\|\frac{\p V_\e}{\p x_1}\right\|_{L_2(\Om_\e)}^2
\\
&+ \frac{1}{16} \|V_\e\|_{L_2(\Om_\e)}^2 + C\e(K+\mu)^2 \left(\k^{-1}+|\ln\e(K+\mu)|^2\right)\|f\|_{L_2(\Om_\e)}^2.
\end{align*}
Therefore, %D%
\begin{equation*}
\|V_\e\|_{L_2(\Om_\e)}\leqslant C\e^{1/2}(K+\mu) \big(\k^{-1/2}
+ |\ln\e(K+\mu)|\big) \|f\|_{L_2(\Om_\e)}.
\end{equation*}
Together with the definition (\ref{4.12}) of $V_\e$, (\ref{4.9}), and (\ref{4.10}) it yields
\begin{align*}
\|U_\e-U_\e^{(\mu)}\|_{L_2(\Om_\e)} &=\|V_\e+U_\e^{(\mu)}(0)W\chi\|_{L_2(\Om_\e)}
\\
&\leqslant C\e^{1/2} (K+\mu) \big(\k^{-1/2} +|\ln\e(K+\mu)|+\e\big)\|f\|_{L_2(\Om_\e)}.
\end{align*}
We combine this inequality with (\ref{4.5}) and it completes the proof of (\ref{1.18}).

The proof of the asymptotics (\ref{1.15}), (\ref{1.17}) is similar to that of Theorem~2.4 in \cite{BBC1}; it is enough to use the proven asymptotics (\ref{1.18}) instead of Theorem~2.3 in \cite{BBC1}.

\section{Bottom of the essential spectrum}

In this section we prove Theorem~\ref{th1.5}.

We begin by proving the identity (\ref{1.19}). We first observe that the form $\hpe(\tau)$ associated with the operator $\Hpe(\tau)$ can be estimated from below as
\begin{equation*}
\hpe(\tau)[u,u]\geqslant \left\|\left(\iu\frac{\p}{\p
x_1}-\frac{\tau}{\e}\right)u\right\|_{L_2(\Om_\e)}^2+ \left\|\frac{\p
u}{\p x_2}\right\|_{L_2(\Om_\e)}^2-|b|\|u\|_{L_2(\Gp_-)}^2,
\end{equation*}
where the  form in the right hand side is on $\Hper^1(\Om_\e,\Gp_+)$. Hence, by the minimax principle we have
\begin{equation}\label{6.2}
\l_1(\e,\tau)\geqslant \inf\limits_{\genfrac{}{}{0 pt}{}{u\in\Hper^1(\Om_\e,\Gp_+)}{\|u\|_{L_2(\Om)}=1}
}\left(  \left\|\left(\iu\frac{\p}{\p
x_1}-\frac{\tau}{\e}\right)u\right\|_{L_2(\Om_\e)}^2+ \left\|\frac{\p
u}{\p x_2}\right\|_{L_2(\Om_\e)}^2 -|b|\|u\|_{L_2(\Gp_-)}^2\right).
\end{equation}
The infimum in the right hand side of this estimate is the lowest eigenvalue of the operator
\begin{equation*}
 \left(\iu\frac{\p}{\p
x_1}-\frac{\tau}{\e}\right)^2-\frac{\p^2}{\p x_2^2}\quad\text{in}\quad L_2(\Om_\e)
\end{equation*}
with periodic boundary condition on the lateral boundaries of $\Om_\e$, the Dirichlet condition on $\Gp_+$, and the Robin condition
\begin{equation*}
\left(\frac{\p}{\p x_2}-b\right)u=0\quad\text{on}\quad\Gp_-.
\end{equation*}
We find this eigenvalue by using the separation of variables and substitute it into (\ref{6.2}),
\begin{equation}\label{6.3}
\l_1(\e,\tau)\geqslant \frac{\tau^2}{\e^2}+T^2,
\end{equation}
where $T$ is the smallest nonnegative root  of the equation
\begin{equation*}
T\cos\pi T+b\sin\pi T=0.
\end{equation*}
It is clear that the first root of (\ref{1.17}) can be estimated uniformly in small $\mu$ as
\begin{equation*}
0\leqslant \sqrt{\L_1(\mu)}\leqslant c,\quad c=\mathrm{const}.
\end{equation*}
Thus, by (\ref{6.3}) for $|\tau|\geqslant (c+1)\e$, $\tau\in[-1,1)$ we have
\begin{equation*}
\l_1(\e,\tau)\geqslant (\sqrt{\L_1(\mu)}+1)^2\geqslant\L_1(\mu)+1.
\end{equation*}
Since by (\ref{1.15})
\begin{equation*}
\l_1(\e,0)=\L_1(\mu)+\Odr(\e^{1/2}),\quad \e\to+0,
\end{equation*}
we conclude that for sufficiently small $\e$ we  have
\begin{equation*}
\inf\limits_{\tau\in[-1,1)} \l_1(\tau,\e)= \inf\limits_{|\tau|\leqslant (c+1)\e} \l_1(\tau,\e).
\end{equation*}
Relation (\ref{1.19}) is now proven by the arguments used in \cite[Sec.5]{BC} %D%
and based on Temple inequalities.

The rest of the section is devoted to construction of the asymptotic expansion for $\l_1(0,\e)$. Here we employ the approach suggested in \cite{AHP-4}, \cite{AHP-23}, \cite{AHP-25}, \cite{AHP-24}.

We write the boundary value problem for the eigenvalue $\l_1(0,\e)$ and the associated eigenfunction $\po(x,\e)$,
\begin{equation}\label{6.4}
\begin{gathered}
-\D\po=\l_1(0,\e)\po\quad\text{in}\quad \Om_\e,
\\
\po=0\quad\text{on}\quad \Gp_+\cup\gp_\e,\qquad \left(\frac{\p
}{\p x_2}-b\right)\po=0\quad \text{on}\quad \Gp_\e,
\end{gathered}
\end{equation}
and on the lateral boundaries of $\Om_\e$ the periodic boundary conditions are assumed. The asymptotics for $\l_1(0,\e)$ is constructed as
\begin{equation}\label{6.5}
\l_1(0,\e)=\L(\mu,\e),
\end{equation}
where the function $\L(\mu,\e)$ is to be determined. By (\ref{1.15}) it should satisfy the identity
\begin{equation}\label{6.1}
\L(\e,\mu)=\L_1(\mu)+\Odr\big(\e^{1/2}(K+\mu)\big), \quad\e\to+0, \quad \mu\to+0.
\end{equation}

The asymptotics for $\po$ is constructed by the combination of the boundary layer method \cite{VL} and the method of matching asymptotic expansions \cite{Il}. It is sought as a sum of an outer expansion, a boundary layer, and the inner expansion. The outer expansion is defined as
\begin{equation}\label{6.6}
\pex(x,\L)=\sin\sqrt{\L(\mu,\e)}(x_2-\pi).
\end{equation}
This function is periodic w.r.t. $x_1$ %D%,
and satisfies the equation and the boundary condition on $\Gp_+$ in (\ref{6.4}) %D%, and there is
no matter how the function $\L$ looks like.

The boundary layer is constructed in terms of the rescaled variables $\xi$, and we denote it as $\pbl(\xi,\mu)$. The main idea of using the boundary layer is to satisfy the required boundary condition on $\Gp_\e$ and thus
\begin{equation*}
\left(\frac{\p}{\p x_2}-b\right)(\pex+\pbl)=0 \quad\text{on}\quad \Gp_\e.
\end{equation*}
We substitute (\ref{6.6}) into this condition, rewrite it in the variables $\xi$, and pass to the limit as $\eta\to+0$. It leads us to one more boundary condition,
\begin{gather}\label{6.7}
\left(\frac{\p}{\p \xi_2}-\e b\right) \pbl=-\e(\sqrt{\L}\cos\sqrt{\L}\pi+b\sin\sqrt{\L}\pi)\quad \text{on}\quad \Gp^0,
\\
\Gp^{0}:=\left\{\xi:
0<|\xi_1|<\frac{\pi}{2},\,\xi_2>0\right\}.\nonumber
\end{gather}
To obtain the equation for $\pbl$, we substitute it and (\ref{6.5}) into (\ref{6.4}) and pass to the variables $\xi$. It yields the boundary value problem
\begin{equation}\label{6.8}
-\D_\xi \psi_\e^{bl}=\e^2\L\psi_\e^{bl},\quad \xi\in\Pi,\qquad
\Pi:=\left\{\xi: |\xi_1|<\frac{\pi}{2},\ \xi_2>0\right\},
\end{equation}
and on the lateral boundaries of $\Pi$ the periodic boundary condition is imposed.

By $\mathfrak{V}$ we denote the space of $\pi$-periodic, even in
$\xi_1$ functions belonging to
$C^\infty(\overline{\Pi}\setminus\{0\})$ and exponentially decaying
as $\xi_2\to+\infty$, together with all their derivatives, uniformly
in $\xi_1$. It is easy to check that the function $X$ introduced in (\ref{3.37}) belongs to $\mathfrak{V}$. We state the next lemma that was proven in \cite{BBC1}.

\begin{lemma}\label{lm6.1}
The function $X$ can be represented as the series
\begin{equation}\label{6.9}
X(\xi)=-\sum\limits_{n=1}^{\infty}
\frac{1}{n}\E^{-2n\xi_2}\cos2n\xi_1,
\end{equation}
which converges in $L_2(\Pi)$ and in $C^k(\overline{\Pi}\cap\{\xi:
\xi\geqslant R\})$ for each $k\geqslant 0$, $R>0$.
\end{lemma}

\begin{lemma}\label{lm6.2}
For small   $\b$ the problem
\begin{equation}\label{6.10}
-\D_\xi Z-\b^2 Z=0\quad\text{in}\quad\Pi,
\quad
\left(\frac{\p}{\p\xi_2}-\e b\right)Z=0\quad\text{on}\quad \Gp^0,
\end{equation}
has an even periodic in $\xi_1$ solution in $\mathfrak{V}$. This solution and
all its derivatives w.r.t. $\xi$ decay exponentially as
$\xi_2\to+\infty$ uniformly in $\xi_1$, $\e$, and $\b$. The differentiable
asymptotics
\begin{equation}\label{6.11}
\begin{aligned}
Z(\xi,\e,\b)=&\E^{\e b\xi_2}\left(\ln|\xi|+\ln 2+\tht(\e b,\b^2)-(1+\e b)\xi_2+\e b\xi_1\left(\vp-\frac{\pi}{2}\right)
\right)
\\
&+\Odr(|\xi|^2\ln|\xi|),\quad\xi\to0,
\end{aligned}
\end{equation}
holds true uniformly in $\b$ and $\e$, where  $\vp$ is the polar angle associated with $\xi$, and the function $\tht$ is defined in (\ref{1.28}). The function $Z$ is bounded in
$L_2(\Pi)$ uniformly in  $\e$ and $\b$.
The function $\tht(t_1,t_2)$ is jointly holomorphic in $t_1$ and $t_2$.
\end{lemma}

\begin{proof}
We make the change
\begin{equation}\label{6.14}
Z(\xi,\e,\b)=\E^{\e b\xi_2} \left(\widetilde{Z}(\xi,\e,\b)+X(\xi)-\e b \xi_2 X-\e b\int\limits_{\xi_2}^{+\infty} X(\xi_1,t)\di t\right)
\end{equation}
in the problem (\ref{6.10}),
\begin{align}
&\left(-\D-2\e b\frac{\p}{\p \xi_2}-\e^2 b^2-\b^2\right) \widetilde{Z}=F\quad\text{in}\quad\Pi,\qquad
\frac{\p \widetilde{Z}}{\p\xi_2}=0\quad\text{on}\quad \Gp^0,\label{6.13}
\\
&F:=(\e^2 b^2+\b^2)X-2\e^2 b^2 \xi_2\frac{\p X}{\p\xi_2}
-\e b (\e^2 b^2+\b^2) \left(\xi_2 X+\int\limits_{\xi_2}^{+\infty} X(\xi_1,t)\di t\right),\nonumber
\end{align}
and on the lateral boundaries of $\Pi$ we have the periodic boundary conditions.

%D%
Let $\mathfrak{V}$ be the orthogonal complement in $L_2(\Pi)$ to the set of the functions $\phi=\phi(\xi_2)$ belonging to $L_2(\Pi)$. By $\mathfrak{W}$ we denote the Hilbert space of functions in $\H^2(\Pi)\cap\mathfrak{V}$ satisfying periodic boundary conditions on the lateral boundaries of $\Pi$, the Neumann boundary condition on $\Gp^0$. %D%, and being orthogonal %D%in $L_2(\Pi)$ to all functions .
In $\mathfrak{V}$ we introduce the operator $\mathcal{B}$ in acting as $-\D$ and on the domain $\mathfrak{W}$. This operator is self-adjoint.

It was shown in the proof of Lemma~5.2 in \cite{BBC1} that $\mathcal{B}\geqslant $ and hence the inverse exists and is bounded.   It follows that $\mathcal{B}^{-1}$ is also bounded as an operator from $\mathfrak{V}$ into $\H^1(\Pi)\cap\mathfrak{V}$. Hence, for sufficiently small $\e$ and $\b$, we have
\begin{equation*}
\left(\mathcal{B}-2\e b\frac{\p}{\p \xi_2}-\e^2 b^2-\b^2\right)^{-1} = \mathcal{B}^{-1} \left(\I-\left(2\e b\frac{\p}{\p \xi_2}+\e^2 b^2+\b^2\right)\mathcal{B}^{-1}
\right)^{-1},
\end{equation*}
%D%
where we have beared in mind that the operator $\frac{\p}{\xi_2}$ maps $\H^1(\Pi)\cap\mathfrak{V}$ into $\mathfrak{V}$.

Thus, the solution to the problem (\ref{6.13}) reads as
\begin{equation*}
\widetilde{Z}=\left(\mathcal{B}-2\e b\frac{\p}{\p \xi_2}-\e^2 b^2-\b^2\right)^{-1}F.
\end{equation*}
By the standard smoothness improving theorems we conclude that $\widetilde{Z}\in C^\infty(\widetilde{\Pi}\setminus\{0\})$.
The function $Z$ can be found from (\ref{6.14}),
\begin{equation*}%\l%abel{6.15}
Z=\E^{\e b\xi_2} \left(\left(\mathcal{B}-2\e b\frac{\p}{\p \xi_2}-\e^2 b^2-\b^2\right)^{-1}F+X-\e b \xi_2 X-\e b\int\limits_{\xi_2}^{+\infty} X(\xi_1,t)\di t\right).
\end{equation*}
We can obtain one more representation for $Z$ by the separation of variables. In order to do it, we construct $Z$ as
\begin{equation*}
Z(\xi,\e,\b)=\widehat{Z}(\xi,\e,\b)+\E^{\e b\xi_2}X(\xi)
\end{equation*}
and then we separate the variables for $\widehat{Z}$. It implies
\begin{equation}\label{6.16}
Z(\xi,\e,\b)=-\sum\limits_{n=1}^{\infty}  \frac{2 \E^{-\sqrt{4n^2-\b^2}\xi_2} }{\sqrt{4n^2-\b^2}+\e b} \cos 2n\xi_1.
\end{equation}
As in the proof of Lemma~5.1 in \cite{BBC1}, we check that this series converges in $L_2(\Pi)$ and in $C^k(\overline{\Pi}\cap\{\xi: \xi_2\geqslant R\})$ for each $k\geqslant 0$, $R>0$. It implies that the function $Z$ and all its derivatives decay exponentially as $\xi_2\to+\infty$ uniformly in $\xi_1$, $\e$, and  $\b$. Hence, $Z\in\mathfrak{V}$ and for sufficiently small $\e$ and $\b$
\begin{align*}
\|Z\|_{L_2(\Pi)}^2=&\sum\limits_{n=1}^{\infty} \frac{2\pi}{(\sqrt{4n^2-\b^2}+\e b)^2}\int\limits_{0}^{+\infty}
\E^{-2\sqrt{4n^2-\b^2}\xi_2}\di \xi_2
\\
=&\sum\limits_{n=1}^{\infty} \frac{2\pi}{(\sqrt{4n^2-\b^2}+\e b)^2\sqrt{4n^2-\b^2}} \leqslant \sum\limits_{n=1}^{\infty} \frac{8\pi}{(4n^2-1)^{3/2}}.
\end{align*}
Hence, the function $Z$ is bounded in $L_2(\Pi)$ uniformly in $\e$ and $\b$.

By analogy with Lemma~3.2 in \cite{AHP-22} one can prove that
\begin{equation}\label{6.17}
\widetilde{Z}(\xi,\e,\b)=\widetilde{Z}(0,\e,\b)+\Odr(|\xi|^2\ln|\xi|),\quad \xi\to0,
\end{equation}
uniformly in small $\e$ and $\b$. Let us calculate $\widetilde{Z}(0,\e,\b)$.

We employ the asymptotics (\ref{3.13}) with $j=0$ and (\ref{6.17}) and we integrate by parts as follows,
\begin{equation*}
\int\limits_{\Pi} X\D \widetilde{Z}\di\xi=\lim\limits_{\d\to0} \int\limits_{\genfrac{}{}{0 pt}{}{|\xi|=\d,}{\xi_2>0}} \left(
\widetilde{Z} \frac{\p X}{\p|\xi|}-X\frac{\p \widetilde{Z}}{\p|\xi|}\right)\di s=\pi \widetilde{Z}(0,\e,\b).
\end{equation*}
Hence, by (\ref{6.14}) and the equation in (\ref{6.13}) we get
\begin{align*}
\widetilde{Z}(0,\e,\b)=&-\frac{1}{\pi} \int\limits_{\Pi} X\left(\left(2\e b\frac{\p}{\p\xi_2}+\e^2 b^2+\b^2\right)\widetilde{Z} +F\right)\di\xi
\\
=&-\frac{1}{\pi} \int\limits_{\Pi} X \left(
\left(2\e b\frac{\p}{\p\xi_2}+\e^2 b^2+\b^2\right)\E^{-\e b\xi_2} Z -2\e b\frac{\p X}{\p\xi_2}\right)\di\xi.
\end{align*}
We substitute the series (\ref{6.9}), (\ref{6.16}) into the obtained equation,
\begin{equation}\label{6.18}
\begin{aligned}
\widetilde{Z}(0,\e,\b)= &\sum\limits_{n=1}^{\infty} \frac{1}{n} \int\limits_{0}^{+\infty}  \bigg(
\frac{2\e b\sqrt{4n^2-\b^2}+\e^2 b^2-\b^2}{\sqrt{4n^2-\b^2}+\e b}\E^{-\xi_2\big(\sqrt{4n^2-\b^2}+2n+\e b\big)}
\\
&\hphantom{\sum\limits_{n=1}^{\infty} \frac{1}{n} \int\limits_{0}^{+\infty}  \bigg( ++ }-2\e b\E^{-4n\xi_2} \bigg)\di\xi_2
\\
=&-\frac{\e b}{2}\sum\limits_{n=1}^{\infty} \frac{1}{n^2} + 2\e b\sum\limits_{n=1}^{\infty} \frac{1}{n\big(\sqrt{4n^2-\b^2}+2n+\e b\big)}
\\
&-(\b^2+\e^2 b^2)
\sum\limits_{n=1}^{\infty}\frac{1}{n\big(\sqrt{4n^2-\b^2}+\e b\big)\big(\sqrt{4n^2-\b^2}+2n+\e b\big)}
\\
=& -\frac{\pi^2\e b}{12}+
\tht(\e b,\b^2).
\end{aligned}
\end{equation}

In view of (\ref{6.14}), to prove (\ref{6.11}), it is required to study also the behavior of the function $\int\limits_{\xi_2}^{+\infty} X(\xi_1,t)\di t$ as $\xi\to0$. We have
\begin{equation}
\int\limits_{\xi_2}^{+\infty} X(\xi_1,t)\di t=\int\limits_{0}^{+\infty} X(\xi_1,t)\di t - \int\limits_{0}^{\xi_2} X(\xi_1,t)\di t.\label{6.19}
\end{equation}
It follows directly from the definition (\ref{3.37}) of $X$ and (\ref{3.36}) that
\begin{align*}
&\frac{d^2}{d\xi_1^2} \int\limits_{0}^{+\infty} X(\xi_1,t)\di t= -\int\limits_{0}^{+\infty} \frac{\p^2}{\p t^2}X(\xi_1,t)\di t=-1,\quad \xi_1\not=0,
\\
&\int\limits_{0}^{+\infty} X\left(\pm\frac{\pi}{2},t\right)\di t=\frac{\pi^2}{24},\quad \int\limits_{0}^{+\infty} \frac{\p X}{\p\xi_2}\left(\pm\frac{\pi}{2},t\right)\di t=0,
\end{align*}
and thus
\begin{equation*}
\int\limits_{0}^{+\infty} X(\xi_1,t)\di t=\frac{\pi^2}{24}-\frac{1}{2}\left(\xi_1\mp \frac{\pi}{2}\right)^2,\quad\pm\xi_1\geqslant 0.
\end{equation*}
The asymptotics (\ref{3.13}) implies
\begin{equation*}
\int\limits_{0}^{\xi_2} X(\xi_1,t)\di t= \xi_2\ln|\xi|+(\ln 2-1)\xi_2+\xi_1\arctan\frac{\xi_2}{\xi_1}+\Odr(|\xi|^2),\quad \xi\to0.
\end{equation*}
The two last equations, (\ref{6.19}), and (\ref{3.13}) yield
\begin{equation*}%\l%abel{6.20}
\xi_2 X(\xi)+\int\limits_{\xi_2}^{+\infty} X(\xi_1,t)\di t=-\frac{\pi^2}{12} +\xi_1\left(\frac{\pi}{2}-\vp\right)-\xi_2+\Odr(|\xi|^2),\quad\xi\to0.
\end{equation*}
This identity, (\ref{6.14}), and (\ref{6.18}) lead us to (\ref{6.11}). The function $\tht$ is jointly holomorphic in $t_1$ and $t_2$ by the first Weierstrass theorem since the terms of series (\ref{1.28}) are jointly holomorphic in $t_1$ and $t_2$ and these series converges uniformly in small $t_1$ and $t_2$.
\end{proof}

The proven lemma allows us to construct the needed solution to the problem (\ref{6.8}), (\ref{6.7}). Namely, we have
\begin{equation*}%\l%abel{6.21}
\pbl(\xi,\e,\L)=\e  (\sqrt{\L}\cos\sqrt{\L}\pi+b\sin\sqrt{\L}\pi)  Z(\xi,\e,\e\sqrt{\L}).
\end{equation*}
By the definition (\ref{6.6}) of $\pex$, it has the asymptotics
%\begin{equation}\label{6.23}
\begin{align*}
\pex(x,\L)=&-\sin\sqrt{\L}\pi+\sqrt{\L}x_2\cos\sqrt{\L}\pi +\Odr(x_2^2)
\\
=&\ \E^{b x_2}\left(-\sin\sqrt{\L}\pi +\big(\sqrt{\L}\cos\sqrt{\L}\pi+b \sin\sqrt{\L}\pi\big)x_2
\right),\quad x_2\to+0.
\end{align*}
%\end{equation}
The asymptotics of $\pbl$ as $\xi\to 0 $ can be easily found by (\ref{6.11}), and, as $\xi\to 0$, we have
\begin{align}
&
\begin{aligned}
\pex(x,\L)+\pbl(\xi,\e,\L)=&\E^{\e \eta b\vs_2}\big(\Psin_0(\vs,\e,\L)+\e\eta\, \Psin_1(\vs,\e,\L)\big)
\\
&+\Odr\left(\e\eta^2|\vs^{(j)}|^2\big(\big|\ln|\vs|\big|+\e^{-1}(K+\mu)\big)\right),
\end{aligned}\nonumber %\l%abel{6.24}
\\
&
\begin{aligned}
\Psin_0(\vs,\e,\L)=&\e\big(\sqrt{\L}\cos\sqrt{\L}\pi+b\sin\sqrt{\L}\pi\big) (\ln|\vs|+\ln 2)
\\
&+\big(\sqrt{\L}\cos\sqrt{\L}\pi+b\sin\sqrt{\L}\pi\big)\big(\e\tht(\e b,\e^2\L)-(K+\mu)^{-1}\big)
\\
&-\sin\sqrt{\L}\pi,
\end{aligned}
\label{6.26}
\\
&
\begin{aligned}
\Psin_1(\vs,\e,\L)=& \e b\big(\sqrt{\L}\cos\sqrt{\L}\pi+b\sin\sqrt{\L}\pi\big) \left(\vs_1\left(\vp-\frac{\pi}{2}\right)-\vs_2\right),
\end{aligned}\label{6.27}
\end{align}
where $\vs=\xi\eta^{-1}$.
In accordance with the method of matching asymptotic expansions we construct the inner expansion for $\po$ in a vicinity of the point $x=0$ as
\begin{equation}\label{6.28}
\mathring{\psi}_\e^{\mathrm{in},j}(\vs,\L)=\E^{\e\eta b\vs_2}\big(\psin_0(\vs,\e,\L)+ \e \eta\,\psin_1(\vs,\e,\L)\big),
\end{equation}
where the functions $\psi_i$, $i=0,1$, must behave at infinity as
\begin{equation}\label{6.29}
\psin_i(\vs,\e,\L)=\Psin_i(\vs,\e,\L)+\Odr(|\vs|^i),\quad \vs\to\infty.
\end{equation}
We substitute the ansatz (\ref{6.28}) into the boundary value problem (\ref{6.3}) and equate the coefficients at the like powers of $\eta$. It implies the boundary value problems for $\psin_i$,
\begin{align}
&\D\psin_0=0,\hphantom{-b\frac{\p\psin_0}{\p\vs_2}} \quad\vs_2>0,\qquad \psin_0=0,\quad\vs\in\go,\quad\frac{\p\psin_0}{\p\vs_2}=0,\quad \vs\in\Go,\label{6.30}
\\
&\D\psin_1=-2b\frac{\p\psin_0}{\p\vs_2}, \quad\vs_2>0,\qquad \psin_1=0,\quad\vs\in\go,\quad\frac{\p\psin_1}{\p\vs_2}=0,\quad \vs\in\Go,\label{6.30a}
\end{align}
where $\go:=\{\vs: |\vs_1|<1,\ \vs_2=0\}$, $\Go:=O\vs_1\setminus\overline{\go}$. The problem (\ref{6.30}) has the only bounded at infinity solution which is trivial. The solution behaving at infinity as $\ln|\vs|$ is also unique and it is the function $Y$ introduced in (\ref{3.42}). Hence,
\begin{equation*}%\l%abel{6.25}
\psin_0(\vs,\e,\L)=\e\big(\sqrt{\L}\cos\sqrt{\L}\pi+b\sin\sqrt{\L}\pi\big) Y(\vs).
\end{equation*}
The asymptotics of this function as $\vs\to\infty$ can be found by using (\ref{3.15}). Comparing this asymptotics with (\ref{6.26}), (\ref{6.29}), we arrive at the equation
\begin{equation*}
 \big(\sqrt{\L}\cos\sqrt{\L}\pi+b\sin\sqrt{\L}\pi\big) \left(\e\tht(\e b,\e^2\L)-(K+\mu)^{-1}\right) -\sin\sqrt{\L}\pi=0,
\end{equation*}
which can be rewritten as (\ref{1.31}). In the same way how similar equation (2.15) was studied in \cite{BBC1}, one can prove easily that the equation (\ref{1.31}) has the unique root satisfying (\ref{6.1}). This root is jointly holomorphic in $\e$ and $\mu$ and it can be represented as the convergent series (\ref{1.32}). To calculate explicitly its coefficients, it is sufficient to substitute this series into the equation (\ref{1.31}) and expand it into the Taylor series w.r.t. $\e$, set the coefficients at the powers of $\e$ equal to zero, and solve the equations obtained. Exactly in this way, one can check the formulas (\ref{1.33}).

To solve the problem (\ref{6.30a}), (\ref{6.27}), we need an auxiliary lemma.

\begin{lemma}\label{lm6.3}
The problem
\begin{equation}\label{6.31}
\D Y_1=-2\frac{\p Y}{\p\vs_2}, \quad\vs_2>0,\qquad Y_1=0,\quad\vs\in\go,\quad\frac{\p Y_1}{\p\vs_2}=0,\quad \vs\in\Go,
\end{equation}
has a solution with the differentiable asymptotics
\begin{equation}\label{6.32}
Y_1(\vs)= \vs_1\left(\vp-\frac{\pi}{2}\right)-\vs_2+c+\Odr(|\vs|^{-1}),\quad \vs\to\infty,
\end{equation}
where $c$ is a constant. This solution belongs to $\H^1(Q)\cap C^\infty(\{\vs: \vs_2\geqslant 0\}\setminus\{\vs: \vs_1=\pm 1,\ \vs_2=0\})$, where $Q$ is any bounded subdomain of $\{\vs: \vs_2\geqslant 0\}$.
\end{lemma}

\begin{proof}
Denote
\begin{equation*}
Y^{(1)}_1(\vs):=\vs_1\IM\ln\big(z+\sqrt{z^2-1}\big)-\IM\sqrt{z^2-1}-\frac{\pi}{2} \vs_1.
\end{equation*}
It is easy to check that this function solves the equation in (\ref{6.31}), satisfies the boundary condition on $\Go$ in (\ref{6.31}) and the asymptotics (\ref{6.32}), and belongs to $\H^1(Q)\cap C^\infty(\{\vs: \vs_2\geqslant 0\}\setminus\{\vs: \vs_1=\pm 1,\ \vs_2=0\})$.

Let $\chi_1=\chi_1(t)$ be an infinitely  differentiable cut-off function with values in $[0,1]$, being one as $t>3$  and vanishing as $t<2$. We construct the solution to the problem (\ref{6.31}), (\ref{6.32}) as
\begin{equation*}%\l%abel{6.33}
Y_1(\vs)=Y^{(1)}_1(\vs)+\chi_1(|\vs|)Y^{(2)}_1(\vs).
\end{equation*}
The function $Y^{(2)}_1$ should solve the problem
\begin{equation}\label{6.34}
\begin{aligned}
&\D Y^{(2)}_1=-2(1-\chi_1)\frac{\p Y}{\p\vs_2}-2\nabla\chi_1\cdot\nabla Y^{(1)}_1-Y^{(1)}_1\D \chi_1, \quad\vs_2>0,
\\
& Y^{(2)}_1=0,\quad\vs\in\go,\quad\frac{\p Y^{(2)}_1}{\p\vs_2}=0,\quad \vs\in\Go,
\\
&Y^{(2)}_1(\vs)=o(|\vs|),\quad \vs\to\infty.
\end{aligned}
\end{equation}
The right hand side in the equation of this problem is compactly supported due to the definition of $\chi_1$ and belongs to  $L_2(\{\vs: \vs_2>0\})\cap C^\infty(\{\vs: \vs_2\geqslant 0\}\setminus\{\vs: \vs_1=\pm 1,\ \vs_2=0\})$.

At the next step we make the Kelvin transform and pass to new variables
\begin{equation*}
y=(y_1,y_2),\quad y_1=\frac{\vs_1}{\vs_1^2+(\vs_2+1)^2},\quad y_2=\frac{\vs_2+1}{\vs_1^2+(\vs_2+1)^2}.
\end{equation*}
Under this transform the boundary value problem for $Y_1^{(2)}$ casts into the boundary value problem for the Poisson equation in the disk $\{y: y_1^2+(y_2-1/2)^2<1/2\}$ with a combination of the homogeneous Dirichlet and Neumann condition. The right hand side in this equation is compactly supported with the support separated from zero and it belongs to $L_2$. Hence, we can apply standard theory of generalized solutions to elliptic boundary value problems, see, for instance, \cite[Ch. I\!V]{Mi}. In accordance with this theory, the obtained boundary value problem is solvable in $\H^1(\{y: y_1^2+(y_2-1/2)^2<1/2\})$. By the smoothness improving theorems, the solution to this problem is infinitely differentiable in a vicinity of zero. Returning back to the problem (\ref{6.34}), we conclude that it is solvable in $\H^1(Q)$ within the class of bounded at infinity functions. Hence, the problem (\ref{6.31}) is solvable and has the solution with the asymptotics (\ref{6.32}). By the smoothness improving theorems the solution belongs $\H^1(Q)\cap C^\infty(\{\vs: \vs_2\geqslant 0\}\setminus\{\vs: \vs_1=\pm 1,\ \vs_2=0\})$.
\end{proof}

In view of the proven lemma the solution to the problem (\ref{6.30a}), (\ref{6.27}) is
\begin{equation*}
\psin_1(\vs,\e,\L)=\e b\big(\sqrt{\L}\cos\sqrt{\L}\pi+b\sin\sqrt{\L}\pi\big) Y_1(\vs). %\l%abel{6.22}
\end{equation*}
The formal construction of the asymptotics is completed.

Denote
\begin{equation*}%\l%abel{6.35}
\Pso(x):=\big(\pex(x,\L)+\chi(x_2)\pbl(\xi,\L)\big) \big(1-\chi\big) + \chi(|\vs|\eta^{1/2})\psin(\vs,\L),
\end{equation*}
where $\L=\L(\e,\mu(\e))$ is the root to the equation (\ref{1.31}). The proof of the next lemma is analogous to that of Lemma~5.3 in \cite{BBC1} and is based on direct calculations.

\begin{lemma}\label{lm6.4}
The function $\Pso$ belongs to $C^\infty(\overline{\Om}_\e\setminus\{x: x_1=\pm\e\eta,\ x_2=0\})$ and to the domain of $\Hpe(0)$. It satisfies the convergence
\begin{equation*}%\l%abel{6.36}
\big\|\Pso-\sin\sqrt{\L_1(0)}(x_2-\pi)\big\|_{L_2(\Om_\e)}\to0,\quad\e\to+0,
\end{equation*}
and solves the equation
\begin{equation*}%\l%abel{6.37}
\big(\Hpe(0)-\L\big)\Pso=h_\e,
\end{equation*}
where for the function $h_\e\in L_2(\Om_\e)$ the uniform in $\e$ estimate
\begin{equation*}%\l%abel{6.38}
\|h_\e\|_{L_2(\Om_\e)}\leqslant C\Big((K+\mu)\E^{-2\e^{-1}}+\eta^{1/2}(K+\mu)+\e^{1/2}\eta^{1/2}(K+\mu)^{1/2}
\Big)
\end{equation*}
holds true.
\end{lemma}

Employing this lemma and proceeding as in \cite[Sec. 5]{BBC1}, one can easily check (\ref{1.34}).

\section*{Acknowledgments}

This work was initiated during the stay of D. Borisov in Universit\'{e} de Lorraine, Metz. He is thankful for the hospitality extended to him.

\end{document}